\documentclass[11pt,a4paper]{article}
\usepackage{amsmath,amssymb,amsthm,epsfig,amscd,verbatim, colortbl}
\usepackage[margin=25mm, tmargin=28mm, twoside]{geometry}
\usepackage[cmtip,arrow]{xy}
\usepackage{pb-diagram,pb-xy}
\usepackage[english]{babel}
\usepackage[T1]{fontenc}
\usepackage[charter]{mathdesign}
\usepackage{srcltx}
\usepackage[colorlinks=true]{hyperref}
\usepackage{colortbl,enumerate,makeidx}
\usepackage[mathscr]{eucal}

\usepackage{ulem}

\newcommand{\point}{\scriptscriptstyle{\bullet}}

\newcommand{\callo}[1]{{\cal L}(\Omega,#1_{\scriptscriptstyle{\bullet}})}

\newcommand{\lun}[1]{L^1(\Omega,#1_{\scriptscriptstyle{\bullet}})}
\newcommand{\callun}[1]{{\cal L}^1(\Omega,#1_{\scriptscriptstyle{\bullet}})}

\newcommand{\calSe}[1]{{\cal S}(\Omega,#1_{\scriptscriptstyle{\bullet}})}

\newcommand{\se}[1]{#1_{\scriptscriptstyle{\bullet}}}
\newcommand{\bd}{\partial}
\newcommand{\ch}[1]{(\Omega,#1_{\scriptscriptstyle{\bullet}})}

\newcommand{\epsi}{\varepsilon}
\newcommand{\ph}{\varphi}

\newcommand{\AQ}{\mathcal{A}_{\mathbf{Q}}}

\newcommand{\CXpoint}{\mathcal{C}(X_{\point})}
\newcommand{\om}[1]{#1_\omega}
\newcommand{\dist}[2]{\se d(\se #1,\se #2)}

\newcommand{\omdansom}[1]{\{\omega\in\Omega\mid#1\}}

\newcommand{\calloom}[2]{{\cal L}(#1,#2_{\scriptscriptstyle{\bullet}})}

\newcommand{\gam}[2]{\gamma_{#1_{\scriptscriptstyle{\bullet}},#2_{\scriptscriptstyle{\bullet}}}}
\newcommand{\gamom}[2]{\gamma_{#1_\omega,#2_\omega}}

\newcommand{\anglal}[3]{\angle_{#1_{\scriptscriptstyle{\bullet}}}
(#2_{\scriptscriptstyle{\bullet}},#3_{\scriptscriptstyle{\bullet}})}

\newcommand{\calloconto}{{\cal L}(\Omega,\mathcal{C}_0(X_{\scriptscriptstyle{\bullet}}))}

\newcommand{\tend}[1]{\rightarrow#1}

\newcommand{\omp}[1]{#1_{\omega'}}

\newcommand{\A}{{\cal A}}
\newcommand{\B}{{\cal B}}
\newcommand{\C}{{\cal C}}
\newcommand{\D}{{\cal D}}

\newcommand{\F}{{\cal F}}
\newcommand{\I}{{\cal I}}

\newcommand{\R}{\mathcal{R}}

\newcommand{\RR}{{\mathbf R}}
\newcommand{\QQ}{{\mathbf Q}}
\newcommand{\NN}{{\mathbf N}}

\newcommand{\nonvide}{\neq\emptyset}
\newcommand{\cat}{CAT(0)}

\DeclareMathOperator{\supess}{sup\,ess}

\DeclareMathOperator{\diam}{diam} \DeclareMathOperator{\rad}{rad}

\newcommand{\UU}{{\mathbf U}}

\newcommand{\ominfini}{\Omega_{\inf=-\infty}}
\newcommand{\ominf}{\Omega_{\inf>-\infty}}
\newcommand{\ommin}{\Omega_{\min}}
\newcommand{\alomomp}{\alpha(\omega,\omega')}
\newcommand{\alompom}{\alpha(\omega',\omega)}

\renewcommand{\I}{\mathcal{I}}
\newcommand{\M}{\mathscr{M}}

\renewcommand{\emptyset}{\text{\O}}

\RequirePackage[hyperpageref]{backref}
   \renewcommand*{\backref}[1]{}
   \renewcommand*{\backrefalt}[4]{
      \ifcase #1
         No cited.
      \or
         Cited on page #2.
      \else
         Cited on pages #2.
      \fi}
\hypersetup{urlcolor=blue,linkcolor=klein,citecolor=violet3,colorlinks=true}
\definecolor{violet1}{rgb}{0.4,0,1}
\definecolor{violet2}{rgb}{0.95,0,0.95}
\definecolor{violet3}{rgb}{0.54,0.30,0.72}
\definecolor{klein}{rgb}{0.13,0,0.71}

\newtheorem{defi}{Definition}[section]
\newtheorem{lem}[defi]{Lemma}
\newtheorem{prop}[defi]{Proposition}
\newtheorem{thm}[defi]{Theorem}

\newtheorem*{thm*}{Theorem}
\theoremstyle{definition}
\newtheorem{rem}[defi]{Remark}
\newtheorem{exemp}[defi]{Example}
\newtheorem{exemps}[defi]{Examples}
\newtheorem{rems}[defi]{Remarks}

\title{\LARGE Actions of Amenable Equivalence Relations on CAT(0) Fields}
\date{}
\author{Martin \textsc{Anderegg} and Philippe P. A. \textsc{Henry}}

\begin{document}
\begin{sloppypar}
\maketitle

\begin{abstract}
\noindent After a short introduction to the general notion of Borel fields of metric spaces we introduce the notion of the action of an equivalence relation on such fields. Then, we specify the study to the Borel fields of proper CAT(0) spaces and we obtain a rigidity result for the action of an amenable equivalence relation on a Borel field of proper CAT(0) spaces. This main theorem is inspired by the result obtained by \textsc{Adams} and \textsc{Ballmann} regarding the action of an amenable group on a proper CAT(0) space.
\end{abstract}

\tableofcontents

\newpage

\section{Introduction}
\subsection{Overview of the Results}
\noindent One of the first link between amenability and negative curvature is certainly the result of \textsc{Avez} \cite{Av70} which states that a compact Riemannian manifold with nonpositive sectional curvature is flat if and only if its fundamental group is of polynomial growth. The amenability here is implicit and it was \textsc{Gromov} \cite{Gro81} who pointed out that in this case the growth of the fundamental group is polynomial if and only if it is an amenable group. After several generalizations obtained by \textsc{Zimmer} \cite{Zim83}, \textsc{Burger} and \textsc{Schroeder} \cite{BS87}, \textsc{Adams} and \textsc{Ballmann} proved the following theorem :

\begin{thm*}[\cite{AB98}]\label{ThmAB}
Let $X$ be a proper CAT(0) space. If $G\subseteq\text{Isom}(X)$ is an amenable group, then at least one of the following two assertions holds.
\begin{enumerate}
\item[(i)] There exists $\xi\in\bd X$ which is fixed by $G$.
\item[(ii)] The space $X$ contains a $G$-invariant flat\footnote{A \textit{flat} is a closed and convex subspace of $X$ which is isometric to $\mathbf{R}^n$ for some $n\in\mathbf{N}$. In particular a point $x\in X$ is a flat of dimension zero.}.
\end{enumerate}
\end{thm*}

\noindent It's interesting for our work to mention that in his article \textsc{Zimmer} also proved another result :

\begin{thm*}[\cite{Zim83}]
Let $\F$ be a Riemannian measurable foliation with transversally (i. e. holonomy) invariant measure and finite total volume. Assume that almost every leaf is a complete simply connected manifold of nonpositive sectional curvature. If $\F$ is amenable, then almost every leaf is flat.
\end{thm*}

\noindent Without getting in all details, a Riemannian measurable foliation has to be understood as an equivalence relation on a measure space such that each equivalence class (leaf) is a smooth manifold endowed with a Riemannian structure that varies in a Borel way. Amenability of the foliation is defined as the amenability of the induced relation on a transversal (a Borel subset of the probability space that meets almost every leaf only countably many times).

\noindent In this paper we study an object close to the one of Riemannian measurable foliation in the context of CAT(0) spaces (or more generally of metric spaces), namely a Borel field of metric spaces. Suppose that a Borel space $\Omega$ is given and that to each $\omega$ we assign a metric space $\om X$. The definition sets what it means for such an assignment to be Borel. It has been studied by many authors (see for example \cite{Cas67}, \cite{Him75}, \cite{DAP76} or \cite{CV77}), often in the particular case when all the $\om X$ are a subspace of a given separable metric space $X$. This notion seems to be the natural one to define the action of an equivalence relation (or more generally of a groupoid). By adapting the techniques of \cite{AB98} to the context of equivalence relations and Borel fields of CAT(0) spaces we managed to prove the following theorem - see the sequel for a precise meaning of the terminology.

\begin{thm}\label{ThmAH}
Let $(\Omega, \mathcal{A}, \mu)$ be a standard probability space and $\R\subseteq\Omega^2$ be an amenable ergodic Borel equivalence relation which quasi preserves the measure. Assume that $\R$ acts by isometries on a Borel field $\ch{X}$ of proper CAT(0) spaces of finite covering dimension. Then at least one of the following assertion is verified.
\begin{enumerate}
\item[(i)] There exists an invariant Borel section of points at infinity $[\se\xi]\in{L}(\Omega,\bd\se X)$.
\item[(ii)] There exists an invariant Borel subfield $(\Omega,\se A)$ such that $\om A\simeq\RR^n$ for almost every $\omega\in\Omega$.
\end{enumerate}
\end{thm}

\noindent This result is a generalization of the result by \textsc{Adams} and \textsc{Ballmann} just as the one by \textsc{Zimmer} generalized the one by \textsc{Avez}. Recently \textsc{Caprace} and \textsc{Lytchak} \cite{CL10} proved a parallel version of \textsc{Adams} and \textsc{Ballmann}'s result by replacing the locally compactness assumption by the one of finite telescopic dimension. Inspired by this result and by the tools we developed, \textsc{Duschesne} \cite{Duc11} managed to prove a version of the last theorem for a Borel field of such spaces.

\noindent The work presented here was done during our PhD (\cite{And10} \& \cite{Hen10}) under the supervision of Nicolas \textsc{Monod}. Both authors would like to thank him warmly for his help, support and inspiration during this period.

\subsection{Basic Definitions and Notations}

\noindent Let $(X,d)$ be a metric space. If $x\in X$ and $r$ is a real number, we use the notation
$$B(x,r)=\{y\in X\mid d(y,x)<r\},\quad\overline B(x,r)=\{y\mid d(x,y)\le r\}\quad\text{and}\quad S(x,r)=\{y\in X\mid d(x,y)=r\}$$
to denote respectively the open, the closed ball and the sphere centered at $x$ of radius $r$. Sometimes we'll also need to use the closure of the open ball that is written $\overline{B(x,r)}$.

\noindent A metric space is called \textit{proper} if all closed balls are compact.

\noindent A map $\gamma : I\rightarrow X$ from an interval $I$ of real numbers to the space $X$ is a \textit{geodesic} if it is isometric. We say that it is a geodesic \textit{segment} if $I$ is compact and a geodesic \textit{ray} if $I=[0,\infty[$. The space $X$ is \textit{geodesic} if every pair of points can be joined by a geodesic.

\noindent  For $x,y\in X$, we denote the image of a geodesic $\gamma:[a,b]\rightarrow X$ such that $\gamma(a)=x$ and $\gamma(b)=y$ by $[x,y]\subseteq X$. A \textit{geodesic triangle} with vertex $x,y,z\in X$ is $\triangle(x,y,z):=[x,y]\cup[x,z]\cup[y,z]$. A \textit{comparison triangle} for $x,y,z\in X$ is an Euclidean triangle $\triangle(\overline{x}, \overline{y}, \overline{z})\subseteq\mathbf{R}^2$ such that $d(x,y)=d(\overline{x}, \overline{y})$, $d(x,z)=d(\overline{x}, \overline{z})$, $d(y,z)=d(\overline{y}, \overline{z})$. It is unique up to isometry. If $q\in[x,y]$, we denote by $\overline q$ the point in $[\overline{x}, \overline{y}]$ such that $d(x,q)=d(\overline x,\overline q)$.

\begin{defi}
The space $X$ is CAT(0) if for every geodesic triangle $\Delta(x,y,z)$ and every point $q\in[x,y]$ the following inequality holds
$$d(x,q)\le d(\overline x,\overline q).$$
\end{defi}

\noindent The general background reference concerning CAT(0) spaces is \cite{BH99} (see also \cite{Bal95}). We'll introduce the various objects and definitions associated to CAT(0) spaces that we need - like the boundary, the projection on a convex subspace, the angles, etc. -  in the section \ref{SecBorelCAT(0)} where we're going to prove that these notions behave ``well'' in the context of Borel fields of CAT(0) spaces.
\smallskip

\noindent Beside CAT(0) spaces, other basic objects that we will consider in this paper are \textit{Borel equivalence relations}. By a \textit{Borel space} we mean a set $\Omega$ equipped with a $\sigma$-algebra $\A$ and we denote it by $(\Omega,\A)$. If $\Omega$ is a completely metrizable separable topological space and $\A$ is the $\sigma$-algebra generated by the open subsets, then $(\Omega,\A)$ is called a \textit{standard Borel space}. A theorem of \textsc{Kuratowski} states that such spaces are all Borel isomorphic provided they're uncountable (see \textit{e.g.} \cite{Kec95}). A standard Borel space together with a probability measure is called a \textit{standard probability space}.

\begin{defi}
Let $(\Omega,\A)$ be a standard Borel space. A Borel equivalence relation $\R$ is a Borel subset $\R\subseteq\Omega^2$ (where $\Omega^2$ is endowed with the product $\sigma$-algebra) which satisfies the following conditions.
\begin{enumerate}
\item[(i)] For every $\omega\in\Omega$, the set $\R[\omega]:=\{\omega'\in\Omega\mid(\omega, \omega')\in\R\}$ called the class of $\omega$ is finite or countably infinite,
\item[(ii)] The set $\R$ contains the diagonal $\Delta_\Omega:=\{(\omega,\omega)\mid\omega\in\Omega\}$, is symmetric in the sense that $\R=\R^{-1}$ (if for $S\subseteq\Omega^2$ we define $S^{-1}:=\{(\omega', \omega)\mid(\omega, \omega')\in S\}$) and satisfies the following transitivity property : if $(\omega, \omega'), (\omega', \omega'')\in\R$, then $(\omega, \omega'')\in\R$.
\end{enumerate}
\end{defi}

\noindent Standard references for equivalence relations are \cite{DJK94}, \cite{FM77a}, \cite{FM77b}, \cite{JKL02}, \cite{Kan08} or \cite{KM04}.

\noindent For a Borel subset $A\subseteq\Omega$, we note $\R[A]:=\cup_{\omega\in A}\R[\omega]$ the \textit{saturation of $A$} which is a Borel set and we say that $A$ is \textit{invariant} if $\R[A]=A$. If a countable group $G$ acts in a Borel way on $\Omega$, then it's action defines naturally the Borel equivalence relation $\R_G$ where $(\omega, \omega')\in\R_G$ if and only if there exists $g\in G$ such that $g\omega=\omega'$. Reciprocally a now classical result due to \textsc{Feldman} and \textsc{Moore} states that each Borel equivalence relation may be obtained by such an action \cite{FM77a}.

\noindent If $(\Omega, \A, \mu)$ is now a standard probability space, we say that $\R$ \textit{quasi preserves the measure} $\mu$ if for every $A\in\A$ such that $\mu(A)=0$ we have $\mu(\R[A])=0$. This is equivalent to the requirement that for each group $G$ such that $\R=\R_G$ the image measures $g_\ast(\mu)$ where $g\in G$, are equivalent to $\mu$ (\textit{i.e.} the measure $g_\ast(\mu)$ is absolutely continuous with respect to $\mu$ and conversely). In case where the measure $\mu$ is invariant by $G$ we say that $\R$ \textit{preserves the measure}. The relation $\R$ is \textit{ergodic} if each Borel saturated set is such that $\mu(A)=0$ or $\mu(A)=1$.

\section{Borel Fields of Metric Spaces}

\subsection{Definitions and First Results}

\noindent A \textit{field of metric spaces} on a set $\Omega$ is a family of metric spaces $\{(\om X,\om d)\}_{\omega\in\Omega}$ indexed by the elements of $\Omega$. The set $\Omega$ is called the \textit{base} of the field and we denote the field by $(\Omega,(\se X,\se d))$ or $(\Omega, \se X)$ or just $\se X$ when the base is implicit. A \textit{section} of the field is the choice of an element of $\om X$ for each $\omega\in\Omega$, so it can be thought as an element of the product $\prod_{\omega\in\Omega}\om X$. We write a section $\se x$ and for each $\omega\in\Omega$ $\om x$ is used to denote the given element of $\om X$. We denote by $\calSe X$ the set of all sections of the field $(\Omega,\se X)$. Given two sections $\se x,\se y\in\calSe X$ we introduce the following \textit{distance function}
$$\begin{array}{cccl}
\se d(\se x,\se y):&\Omega&\mapsto&[0,\infty[\\
&\omega&\mapsto&\om d(\om x,\om y).
\end{array}$$

\noindent Suppose now that $(\Omega,\A)$ is a Borel space.

\begin{defi}\label{DefiBorelField}
Let $(\Omega,{\cal A})$ be a Borel space and $\ch X$ be a field of metric spaces on $\Omega$. A Borel structure on $\ch X$ is a subset $\callo X\subseteq\calSe X$ such that
\begin{itemize}
\item[(i)] (Compatibility) For all $\se x, \se y\in\callo X$ the function $\dist{x}{y}$ is Borel.
\item[(ii)] (Maximality) If $\se y\in\calSe X$ is such that $\se d(\se x,\se y)$ is Borel for all $\se x\in\callo X$, then $\se y\in\callo X$.
\item[(iii)] (Separability) There exists a countable family ${\cal D}:=\{\se {x^n}\}_{n\ge1}\subseteq\callo X$ such that $\overline{\{\om{x^n}\}}_{n\ge1}=X_\omega$ for all $\omega\in\Omega$. Sometimes we'll write $\om\D:=\{\om{x^n}\}_{n\ge1}$.
\end{itemize}
If there exists such a $\callo X$, we say that $\ch X$ is a Borel field of metric spaces and $\callo X$ is called the Borel structure of the field. The elements of $\callo X$ are called the Borel sections. A set ${\cal D}$ satisfying the condition (iii) is called a fundamental family of the Borel structure $\callo X$.
\end{defi}

\begin{rems}\quad
\smallskip

\noindent (1) Observe that the Condition \ref{DefiBorelField} (iii) forces all the metric spaces $\om X$ to be separable.
\smallskip

\noindent (2) It will follow from Lemma \ref{LemmaBorelGluingPoncutalLimit} and Lemma \ref{LemmaBorelStructureFundamentalFamily} that if a field is \textit{trivial}, \textit{i.e.} all the $\om X$ are the same separable metric space $X$, then the set of all Borel functions from $\Omega$ to $X$ is naturally a Borel structure on this field. This observation should reinforce the intuition of thinking about the Borel sections as a replacement of the Borel functions which cannot be defined when the field is not trivial.
\smallskip

\noindent (3) If $\Omega'\subseteq\Omega$ is a Borel subset, then $\mathcal L(\Omega',\se X):=\{(\se x\mid_{\Omega'})\mid\se x\in\callo X\}$ is a Borel structure on the field $(\Omega',\se X)$, where $\se x\mid_{\Omega'}$ denotes the section of ${\cal S}(\Omega',\se X)$ obtained by restricting the section $\se x$ to the subset $\Omega'$.
\smallskip

\noindent (4) Borel fields of metric spaces can also be presented as bundles (see \textit{e.g.} \cite{DAP76}).
\end{rems}

\noindent We now describe two constructions that we will use to give a useful reformulation of the maximality condition of the
Definition \ref{DefiBorelField}. Let $\{\se x^n\}_{n\ge1}$ be a sequence of elements of $\calSe X$ such that $\{\om x^n\}_{n\ge1}$ is a converging sequence in $\om X$ for all $\omega\in\Omega$. Then we can define a new section $\se x\in\calSe X$ by $\om x:=\lim_{n\tend\infty}\om x^n$ for all $\omega\in\Omega$. This section is called the \textit{pointwise limit} of the sequence $\{\se x^n\}_{n\ge1}$ and it is written $\lim_{n\tend\infty}\se x^n$. Let $\{\se x^n\}_{n\ge1}$ be a sequence of elements of $\calSe X$ and $\Omega=\sqcup_{n\ge1}\Omega_n$ be a countable Borel partition of $\Omega$. Define a new section $\se x$ by setting $\se x\mid_{\Omega_n}:=\se x^n\mid_{\Omega_n}$ for all $n\ge1$. This new section is called a \textit{countable Borel gluing} of the sequence with respect to the partition. If the partition is finite, we call the section a \textit{finite Borel gluing}.

\begin{lem}\label{LemmaBorelGluingPoncutalLimit}
Let $(\Omega,\A)$ be a Borel space and $\ch X$ be a field of metric spaces. Suppose that the set $\callo X\subseteq\calSe X$ is such that the conditions (i) and (iii) of the Definition \ref{DefiBorelField} are satisfied. Then the condition (ii) of the same definition is equivalent to
\begin{itemize}
\item[(ii)'] $\callo X$ is closed under pointwise limits and countable Borel gluings,
\end{itemize}
or to
\begin{itemize}
\item[(ii)'']  $\callo X$ is closed under pointwise limits and finite Borel gluings.
\end{itemize}
\end{lem}

\begin{proof} We will prove (ii) $\Rightarrow$ (ii)'' $\Rightarrow$ (ii)' $\Rightarrow$ (ii).
\smallskip

\noindent [(ii) $\Rightarrow$ (ii)''] This assertion follows easily by applying to the distance functions the facts that a limit of a pointwise converging sequence of Borel functions is still Borel and that a countable Borel gluing of Borel functions is again a Borel function.
\smallskip

\noindent [(ii)'' $\Rightarrow$ (ii)']  Assume that $\se x\in\calSe X$ is the gluing of the sequence $\{\se y^n\}_{n\ge1}\subseteq\callo X$ relatively to the decomposition $\Omega=\sqcup_{n\ge1}\Omega_n$. For all $n\ge1$ we introduce the section $\widetilde{\se y}^n$ defined by
$$\widetilde{\se y}^n\mid_{\Omega_j}:=\begin{cases}\se y^j\mid_{\Omega_j}&\text{if }1\le j\le n,\\
\se y^1\mid_{\Omega_j}&\text{if }j\ge n+1.
\end{cases}$$By hypothesis, $\widetilde{\se y}^n\in\callo X$ for all $n\ge1$ and we have $\lim_{n\tend\infty}\widetilde{\se y}^n=\se x$ so that $\callo X$ is closed under countable Borel gluing.
\smallskip

\noindent [(ii)' $\Rightarrow$ (ii)] Suppose that $\se y\in\calSe{X}$ is such that $\se d(\se x,\se y)$ is Borel for all $\se x\in\callo X$. Fix $\D:=\{\se x^n\}_{n\ge1}$ a fundamental family and define for all $k\ge1$ a function $\se n^k:\Omega\tend\NN$ by
$$\om n^k:=\min\{n\in\NN\mid\ \om d(\om x^n,\om y)\le 1/k \}.$$
Those functions are well defined because $\om\D$ is dense and Borel because
$$\omdansom{\om n^k\le N}=\cup_{j=1}^N \big(\se d(\se x^j,\se y)\big)^{-1}([0,1/k])\in\A\quad\text{for all }N\ge1.$$
For all $k\ge1$, we can define a section $\se{x^{\se n^k}}\in\calSe X$ by gluing the sequence $\{\se x^n\}$ in this way :
$$\se {x^{\se n^k}}\mid_{\big(\se n^k\big)^{-1}(\{j\})}:=\se{x^j}\mid_{\big(\se n^k\big)^{-1}(\{j\})}\quad\text{for all }\ j\ge1.$$
By hypothesis and construction $\{\se x^{\se n^k}\}_{k\ge1}\subseteq\callo X$ and $\lim_{k\rightarrow\infty}\se{x^{\se n^k}}=\se y$, so we can conclude that $\se y\in\callo X$.
\end{proof}

\noindent The following lemma gives two characterizations of the Borel sections by knowing only a fundamental family.

\begin{lem}\label{LemmaBorelStructureFundamentalFamily}
Let $(\Omega,\A)$ be a Borel set, $\ch X$ be a Borel field of metric spaces of Borel structure $\callo X$ and $\D$ be a fundamental family. Then
\begin{eqnarray*}
\callo X&=&\{\se y\in\calSe X\mid \se d(\se y,\se z)\text{ is Borel for every }\se z\in\D\}\\
&=&\{\se y\in\calSe X\mid \se y\text{ is a pointwise limit of countable Borel gluings of elements of $\D$}\}.
\end{eqnarray*}
\end{lem}

\begin{proof}
Let's prove the first equality. The inclusion $[\subseteq]$ is obvious. For the reverse, suppose that $\se y$ is in the right-hand set. Since $\D$ is a fundamental family the equality
$$\se d(\se x,\se y)=\sup_{\se z\in\D}|\se d(\se x,\se z)-\se d(\se z,\se y)|$$
holds for every $\se x\in\callo X$. Therefore $\se d(\se x,\se y)$ is a Borel function and therefore $\se y\in\callo X$. Note that the second equality was already verified in the proof of the Lemma \ref{LemmaBorelGluingPoncutalLimit}.
\end{proof}

\begin{exemps}\label{ExampBorelField}\quad
\smallskip

\noindent (1) As already said a trivial field is a natural example of a Borel field of metric spaces. It is important to keep in mind that, even in the trivial case, many different Borel structures may exist on the same field.
\smallskip

\noindent (2) A standard bundle (in the sense of \textsc{Gaboriau} and \textsc{Alvarez}, see \textit{e.g.} \cite{Alv08}) is a standard Borel space $X$ with a Borel projection $\pi: X\rightarrow\Omega$ such that fibers are countable. The field $\{(\pi^{-1}(\omega),\om d)\}_{\omega\in\Omega}$ where $\om d$ is the discrete distance is a Borel field when endowed with the structure $\{\se f:\Omega\tend X\mid \se f \text{ is Borel and }\pi(\om f)=\omega\}$. (A selection theorem can be used to construct a fundamental family, see \cite{Alv08}.)
\smallskip

\noindent (3) A Borel equivalence relation on $\Omega$ is a particular example of a standard Bundle. This example can be turned into in a more interesting one if the relation is graphed, so that we can consider on each equivalence class the metric induced by the graph structure instead of the discrete one (see \textit{e.g.} \cite{Gab00} for a definition of a graphed equivalence relation).
\smallskip

\noindent (4) A Borel field of Hilbert space as defined in \cite{Dix69} or a Borel field of Banach spaces as defined in \cite{AR00} are examples of Borel field of metric spaces.
\smallskip

\noindent (5) Suppose that there exists a countable family $\D=\{\se x^n\}_{n\ge1}\subseteq\calSe X$ such that $\se d(\se x^n,\se x^m)$ is Borel for every $n,m\ge1$ and $\{\om x^n\}_{n\ge1}$ is dense in $\om X$ for every $\omega\in\Omega$. Then it's easy to adapt the proof of Lemma \ref{LemmaBorelStructureFundamentalFamily} to show that
$$\mathcal L_{\D}(\Omega,\se X):=\{\se y\in\calSe X\mid \se d(\se x,\se z)\text{ is Borel for every }\se z\in\D\}$$
is a Borel structure on $(\Omega,\se X)$.
\end{exemps}

\begin{defi}
Let $(\Omega,{\cal A})$ be a Borel space, $\ch X$ and $\ch Y$ be two Borel fields of respective Borel structures
$\callo X$ and $\callo Y$.

\noindent A morphism between this two Borel fields is a family of applications $\se\varphi=\{\om\varphi:\om X\tend\om
Y\}_{\omega\in\Omega}$ such that for all $\se x\in\callo X$, the section defined by $\se\varphi(\se x):\omega\mapsto\om\varphi(\om x)$ is in $\callo Y$. Sometimes we will simply write $\se\varphi:(\Omega,\se Y)\tend(\Omega,\se Y)$ or $\se\varphi:\se X\tend\se Y$. We denote by $\widetilde{\mathcal{L}}\big(\Omega,\F(\se X,\se Y)\big)$ the set of morphisms from $\ch X$ to $\ch Y$.

\noindent A morphism $\se\varphi$ is called continuous (isometric, injective, surjective, bijective) if $\om\varphi$ is continuous (isometric, injective, surjective, bijective) for every $\omega\in\Omega$. We'll write $\widetilde{\mathcal{L}}\big(\Omega,\C(\se X,\se Y)\big)$ for the set of continuous morphisms and $\widetilde{\mathcal{L}}\big(\Omega,\mathcal{I}(\se X,\se Y)\big)$ for the set of
isometric ones.

\noindent A continuous morphism $\se\ph\in\widetilde{\mathcal{L}}\big(\Omega,\C(\se X,\se Y)\big)$ is invertible if $\om\ph$ is an homeomorphism for every $\omega\in\Omega$ and if $\se\ph^{-1}$ (defined in the obvious way) is also a morphism.
\end{defi}

\begin{rems}\quad
\smallskip

\noindent (1) Obviously we could define a morphism for fields with different bases, but it won't be relevant in our context.
\smallskip

\noindent (2) To verify that a family of continuous applications $\{\om\varphi:\om X\rightarrow\om Y\}$ is a morphism, it's enough to check that $\se\ph(\D)\subseteq\callo Y$ for a fundamental family $\D$ of $\callo X$. Indeed it's an easy consequence of the Lemma \ref{LemmaBorelStructureFundamentalFamily}. In the same spirit, we can verify that $\se\ph\in\widetilde{\mathcal{L}}\big(\Omega,\C(\se X,\se Y)\big)$ is invertible if and only if $\om\varphi$ is an homeomorphism for every $\omega\in\Omega$.
\end{rems}

\noindent Suppose now that we choose for every $\omega\in\Omega$ a subset (possibly empty) $\om A$ of $\om X$. Such a choice is called a \textit{subfield} and we would like to define when it is Borel. A natural way of doing so is to suppose that $\Omega':=\omdansom{\om A\nonvide}$ is a Borel subset of $\Omega$ and that $\mathcal L(\Omega',\se A):=\mathcal L(\Omega',\se X)\cap\mathcal S(\Omega',\se A)$ is a Borel structure on the field\footnote{Observe that $(\Omega,\se A)$ is \textit{not} a field of metric spaces !} $(\Omega',\se A)$. As $\mathcal L(\Omega',\se X)\cap\mathcal S(\Omega',\se A)$ is closed under Borel gluings and pointwise limits\footnote{In this context we think about $(\Omega',\se A)$ as a field, so that we are only interested in pointwise limits that are in $\mathcal S(\Omega',\se A)$. It doesn't mean that the set we consider is closed in $\mathcal S(\Omega',\se X)$ or in $\mathcal L(\Omega',\se X)$}, and as condition (i) is obviously satisfied, Lemma \ref{LemmaBorelGluingPoncutalLimit} naturally leads to the following definition.

\begin{defi}\label{DefiBorelSubfield}
Let $(\Omega,\A)$ be a Borel space and $\ch X$ a Borel field of metric spaces. A subfield $(\Omega,\se A)$ is called Borel if
\begin{enumerate}
\item[(i)] $\Omega':=\omdansom{\om A\nonvide}\in\A$.
\item[(ii)] There exists a countable family of sections $\D'=\{\se y^n\}_{n\ge1}\subseteq\mathcal L(\Omega',\se A)$ such that $\om A\subseteq\overline{\{\om y^n\}}_{n\ge1}$ for every $\omega\in\Omega'$.
\end{enumerate}
The set $\Omega'$ is called the base of the subfield and $\D'$ is called a fundamental family of the subfield.
\end{defi}

\begin{rems}\label{RemFieldBetweenAndUnion}\text{}
\smallskip

\noindent(1) $(\Omega',\se A)$ is a Borel field of metric spaces.
\smallskip

\noindent (2) In the condition (ii) of the previous definition, the closure $\overline{\{\om y^n\}}_{n\ge1}$ is taken in $\om X$, that's why we used $\subseteq$ and not an equality. An obvious way to construct a Borel subfield is to take a countable family $\{\se x^n\}_{n\ge1}\subseteq\callo X$ and to choose a subfield $\se A$ such that $\{\om x^n\}_{n\ge1}\subseteq\om A\subseteq\overline{\{\om x^n\}}_{n\ge1}$ for every $\omega\in\Omega$. In particular a Borel section is an obvious example of a Borel subfield.
\smallskip

\noindent (3) The previous construction can be generalized in the following way : if $\se A$ is a Borel subfield of $\se X$ and $\se B$ is a subfield such that $\om A\subseteq\om B\subseteq\overline{\om A}$ for every $\omega\in\Omega$, then $\se B$ is also a Borel subfield of $\se X$.
\smallskip

\noindent (4) Suppose that $\{\se A^n\}_{n\ge1}$ is a family of Borel subfields. Then the subfield $\se A:=\cup_{n\ge1}\se A^n$, defined by $\om A:=\cup_{n\ge1}\om A^n$, is also a Borel subfield. This can be shown in three steps. First we observe that the base $\Omega'$ of $\se A$ is $\cup_{n\ge1}\Omega_n\in\A$ where $\Omega_n$ denotes the base of $\se A^n$. Then we construct a section $\se z\in\mathcal L(\Omega', \se A)$ : pick sections $\se z^1\in\mathcal L(\Omega_1,\se A^1)$, $\se z^2\in\mathcal L(\Omega_2\setminus\Omega_1,\se A^2)$, $\se z^3\in\mathcal L(\Omega_3\setminus(\Omega_1\cup\Omega_2), \se A^3)$ and so on; gluing them together gives the desired section. Finally we choose, for each $n\ge1$, a fundamental family $\D^n$ of $\se A^n$ and we modify each of its elements by gluing it with $\se z\mid_{\Omega'\setminus\Omega_n}$ to obtain a subset $\widetilde\D^n$ of $\mathcal L(\Omega',\se A)$. By construction $\D':=\cup_{n\ge1}\widetilde\D^n$ is a fundamental family of $\se A$.
\smallskip

\noindent (5) Observe the following obvious property of transitivity for Borel subfields. Let $(\Omega,\A)$ be a Borel space, $\ch X$ be a Borel field of metric spaces, $(\Omega,\se A)$ be a Borel subfield of $\ch X$ of base $\Omega'$ and $\ch B$ be a subfield of $\ch X$ such that $\om B\subseteq\om A$ for all $\omega\in\Omega$. Then $\ch B$ is a Borel subfield of $\ch X$ if and only if $(\Omega',\se B)$ is a Borel subfield of $(\Omega',\se A)$.
\end{rems}

\subsection{Equivalence Classes}

\noindent Suppose now that a Borel probability measure $\mu$ on $(\Omega,\A)$ is given. In this context we can define the equivalence relation of equality almost everywhere on the set of Borel sections and on the set of Borel subfields. Two sections (or two Borel subfields) are equal almost everywhere if the set where they differ is of measure 0. We write $\se x=_{\mu\text{-a.e.}}\se y$ (or $\se A=_{\mu\text{-a.e.}}\se B$) when two sections are equivalent and we denote by $[\se x]$ (or $[\se A]$) the equivalence class of $\se x$ (respectively of $\se A$). We write $L(\Omega,\se X)$ for the set of equivalence classes of Borel sections.
\smallskip

\subsection{Borel Subfields of Open Subsets}

\noindent In the case of a subfield such that all the subsets are open, there is an easy sufficient criterion to verify that it is Borel.

\begin{lem}\label{LemmaSubfieldOpenSubsets}
Let $(\Omega,\A)$ be a Borel space, $\ch X$ be a Borel field of metric spaces and $\se U$ be a subfield such that $\om U$ is open for every $\omega\in\Omega$. If
$\omdansom{\om x\in\om U}\in\A$ is Borel for all sections $\se x$ in a fundamental family of the Borel structure $\callo X$, then $\se U$ is a Borel subfield.
\end{lem}

\begin{proof}
Write $\D:=\{\se x^n\}_{n\ge1}$ a fundamental family and $\Omega_n:=\omdansom{\om x^n\in U_{\omega}}\in\A$. By density of $\{\om x^n\}_{n\ge1}$ and since $\om U$ is open for all $\omega\in\Omega$, we have that $\Omega':=\omdansom{\om
U\nonvide}=\cup_{n\ge1}\Omega_n$ is Borel. For all $n\ge1$, we can define a Borel subfield $(\Omega,\se A^n)$ by
$$\om A^n:=\begin{cases} \{\om x^n\} &\text{if }\omega\in\Omega_n\\
\emptyset&\text{else.}
\end{cases}$$
Then by construction $\cup_{n\ge1}\se A^n\subseteq\se U\subseteq\overline{\cup_{n\ge1}\se A^n}$ and so $\se U$ is a Borel subfield by the Remarks \ref{RemFieldBetweenAndUnion} (3) and (4).
\end{proof}

\begin{exemp}\label{ExampOpenBall}
If $\se x\in\callo X$ and $\se r$ is a Borel non-negative function, then the subfield of open balls $B(\se x,\se r)$ defined by assigning to each $\omega$ the set $B(\om x,\om r)$ is Borel. In fact if $\se y\in\callo X$ then
$$\omdansom{\om y\in B(\om x,\om r)}=\big(\se d(\se x,\se y)\big)^{-1}([0,\se r])\in\A.$$
The field $\overline{B(\se x,\se r)}$ of the closure of the open balls is also Borel because of the Remark \ref{RemFieldBetweenAndUnion} (3).
\end{exemp}

\subsection{Borel Subfields of Closed Subsets}

\noindent If every metric space $\om X$ is complete, then there is a sufficient and necessary criterion for a subfield of closed subsets to be Borel. We choose the convention that the distance from a point to the empty set is infinite.

\begin{prop}\label{PropDistanceBorel}
Let $(\Omega,\A)$ be a Borel space, $\ch X$ be a Borel field of complete metric spaces and $\se F$ be a subfield of closed subsets. Then the following assertions are equivalent.
\begin{itemize}
\item[(i)] $\se F$ is a Borel subfield,
\item[(ii)] $\se d(\se x,\se F):\Omega\tend\RR_+\cup\{\infty\}$, $\omega\mapsto\om d(\om x,\om F)$ is a Borel map for every $\se x\in\callo X$,
\item[(iii)] $\se d(\se x,\se F):\Omega\tend\RR_+\cup\{\infty\}$, $\omega\mapsto\om d(\om x,\om F)$ is a Borel map for every $\se x\in\D$ where $\D$ is a fundamental family of $\callo X$.
\end{itemize}
\end{prop}

\begin{proof}[On the proof]
[(iii) $\Rightarrow$ (ii)] Since the distance to a set is a continuous function, this assertion is a consequence of the Lemma \ref{LemmaBorelStructureFundamentalFamily}.
\smallskip

\noindent The implication [(iii) $\iff$ (i)] was proved in the particular case of trivial fields by \textsc{Castaing} and \textsc{Valadier} \cite{CV77}. Notice that $\omdansom{\om F\nonvide}=\big(\se d(\se x,\se F)\big)^{-1}(\RR_+)$ for any $\se x\in\callo X$ so that we can suppose without lost of generality that this set is equal to $\Omega$. The trivialization theorem due to \textsc{Valadier} (see \cite{Val78} and the remark below), the implication [(ii) $\iff$ (iii)], the completeness assumption and the property of transitivity of Borel subfields (\textit{cf.} Remark \ref{RemFieldBetweenAndUnion} (5)) can therefore be combined to extend the result to the general case.
\end{proof}

\begin{rem}
The trivialization theorem due to \textsc{Valadier} states that for every Borel field of metric spaces $(\Omega,\se X)$, there exists an isometric morphism $\se\ph:(\Omega,\se X)\rightarrow(\Omega,\UU)$ where $\UU$ is the universal separable metric space constructed by \textsc{Uryshon} \cite{Ury27}. In his paper \textsc{Valadier} checks that the isometric embedding of a separable metric space $X$ in $\UU$ can be done in a Borel way.
\end{rem}

\begin{exemp}
Suppose that every $\om X$ is geodesic. If $\se x$ is a Borel section and $\se r:\Omega\tend\RR_+$ is a Borel function, then the field of closed balls $\overline B(\se x,\se r)$ and the one of spheres $S(\se x,\se r)$ defined similarly as in the Example \ref{ExampOpenBall} are Borel. Indeed if $\se y\in\callo X$, then
$$\se d(\se y,\overline B(\se x,\se r))=\lceil\se d(\se y,\se x)-\se r\rceil^0\text{ et }\se d(\se y,S(\se x,\se r))=|\se d(\se y,\se x)-\se r|$$
where if $\alpha\in\RR$, then $$\lceil \alpha\rceil^0=\begin{cases}\alpha&\text{if }\alpha\ge0\\0&\text{if }\alpha<0.\end{cases}$$
\end{exemp}

\noindent In the measure case, we can show that the set of equivalence classes of a Borel subfield of closed sets can be turned into a complete lattice (\textit{i.e.} every subset has an \textit{infimum} and a \textit{supremum}) when it is endowed with the following order : if $\se F$ and $\se G$ are Borel subfields of closed sets, then $[\se F]\le[\se G]$ if $\om F\subseteq \om G$ for almost every $\omega\in\Omega$.

\noindent We'll need the following proposition which can be deduced from \cite{Him75} (Theorem 3.5 and the explanation at the beginning of the Section 4) as the Proposition \ref{PropDistanceBorel} has been deduced from \cite{CV77}.

\begin{prop}\label{PropIntersectionClosedField}
Let $(\Omega,\A,\mu)$ be a standard probability space and $\se X$ be a Borel field of complete metric spaces. Let $\{\se F^n\}_{n\ge1}$ be a family of Borel subfields of closed subsets. Then there exists a Borel subset $\Omega_0$ of measure one such that if $\cap_{n\ge1}\se F^n$ is defined by assigning to each $\omega$ the set $\big(\cap_{n\ge1}\se F^n\big)_\omega:=\cap_{n\ge1}\om F^n $, then $(\Omega_0,\cap_{n\ge1}\se F^n)$ is a Borel subfield of $(\Omega_0,\se X)$.
\end{prop}

\begin{thm}\label{ThmBorelSubfieldCompleteLattice}
Let $(\Omega,\A,\mu)$ be a standard probability space and let $\se X$ be a Borel field of complete metric spaces. Then the set of equivalence classes of Borel subfield of closed subsets with the order of inclusion almost everywhere is a complete lattice.

\noindent More precisely, if $\{[\se F^\beta]\}_{\beta\in \B}$ is a family of equivalence classes of Borel subfields of closed subsets, then there exists a sequence of indices $\{\beta_n\}_{n\ge1}\subseteq\B$ such that
$$[\cap_{n\ge1}\se F^{^{\beta_n}}]\quad\text{and}\quad [\overline{\cup_{n\ge1}\se F^{^{\beta_n}}}]$$
are respectively the \textit{infimum} and the \textit{supremum} of the family $\{[\se F^\beta]\}_{\beta\in\B}$.
\end{thm}

\noindent Before proving this theorem, we recall the notion of an essential supremum of a family of Borel functions, whose existence is guaranteed by the following theorem.

\begin{thm}[\cite{Doo94}, p. 71]\label{ThmSupEss}
Let $(\Omega,\A,\mu)$ be a standard probability space with $\mu$ $\sigma$-finite. Let $\{f_i : \Omega\rightarrow\mathbf{R} \cup\{\pm\infty\}\}_{i\in\I}$ be a family of Borel functions. Then there exists a Borel function $g:\Omega\rightarrow\RR\cup\{\pm\infty\}$ such that
\begin{enumerate}
\item[(i)] For all $i\in\I$, we have $g(\omega)\ge f_i(\omega)$ for almost every $\omega\in\Omega$.
\item[(ii)] If $h : \Omega\rightarrow\mathbf{R}\cup\{\pm\infty\}$ is a Borel function satisfying (i), then $h(\omega)\ge g(\omega)$ for almost every $\omega\in\Omega$.
\end{enumerate}
The function $g$ is uniquely determined up to null sets and all functions in its class satisfy (i) and (ii). Moreover there exists a countable family of elements of $I$ such that its \textit{supremum} satisfies (i) and (ii).

\noindent We call $g$ an essential \textit{supremum} of the family $\{f_i\}_{i\in I}$ and we write $g=\supess_{i\in I}\{f_i\}$.
\end{thm}

\begin{proof}[Proof of the Theorem \ref{ThmBorelSubfieldCompleteLattice}]
First observe the following criterion for a closed subset to be included in another one. If $X$ is a metric space, $D\subseteq X$ is a dense subset and $F_1, F_2\subseteq X$ are two closed subsets, then
\begin{equation}\label{ObsInclusion}
F_2\subseteq F_1\iff d(x, F_2)\ge d(x, F_1)\ \text{for all}\ x\in D.
\end{equation}
Now let $\{\se F^\beta\}_{\beta\in\B}$ be a family of Borel subfields of closed subsets and fix $\D=\{\se x^i\}_{i\ge1}$ a fundamental family of the Borel structure $\callo X$. By Theorem \ref{ThmSupEss} there exists, for each $i\ge1$, a sequence of indices $\{\beta_n^i\}_{n\ge1}$ such that $\sup_{n\ge1}\se d(\se x^i,\se F^{\beta^i_n})$ is an essential \textit{supremum} of the family of functions $\{\se d(\se x^i,\se F^{\beta})\}_{\beta\in\B}$. If we set $\B':=\{\beta_n^i\}_{n,i\ge1}$ then we can simultaneously construct an essential \textit{supremum} for each family $\{\se d(\se x^i,\se F^{\beta})\}_{\beta\in\B}$ by taking $\sup_{\beta\in\B'}\se d(\se x^i,\se F^\beta)$. By Proposition \ref{PropIntersectionClosedField}, there exists a Borel subfield $\se F$ of closed subsets such that $\om F=\cap_{\beta\in\B'}\om F^{\beta}$ for almost every $\omega\in\Omega$, and we'll show that it is the \textit{infimum} of the family $\{\se F^\beta\}_{\beta\in\B}$. Let $\beta_0\in\B$ be fixed. Then for every $i\ge1$ we have
$$\se d(\se x^i,\se F)=_{a.e.}\se d\big(\se x^i,\cap_{\beta\in\B'}\se F^{\beta}\big)\stackrel{(\ref{ObsInclusion})}{\ge}\sup_{\beta\in\B'}\se
d(\se x^i,\se F^\beta)\ge_{a.e.}\se d(\se x^i,\se F^{\beta_0}),$$ which shows, by the preliminary observation, that $\om F\subseteq\om F^{\beta_0}$ for almost every $\omega\in\Omega$. Thus $[\se F]$ is a minorant and it's obvious from its definition that it is the biggest one.

\noindent The same argument can be done for the \textit{supremum} by considering the essential \textit{infimum} of the families $\{\se d(\se x^i,\se F^{\beta})\}_{\beta\in\B}$ realized as the \textit{infimum} taken over a subset $\B''\subseteq\B$. Then $[\overline{\cup_{\beta\in\B''}F^{\beta}}]$ will be the \textit{supremum}. To have the exact formulation of the conclusion of the theorem we only have to order the countable set $\B'\cup\B''=\{\beta_n\}_{n\ge1}$.
\end{proof}

\subsection{Borel Fields of Proper Metric Spaces}

In all this section $\se X$ will denote a Borel field of proper metric spaces. We'll show that the field assigning to each $\omega$ the space of continuous functions on $\om X$ is a Borel field of metric spaces. To do so, we'll need the following lemma.

\begin{lem}\label{LemmaSpaceContinuousFunctionProperSpace}
Let $X$ be a proper metric space and $x_0$ a fixed base point. Then the following function is a metric on $\C(X)$ that induced the topology of uniform convergence on compact sets
$$\begin{array}{cccl}
\delta:&\C(X)\times\C(X)&\rightarrow&\RR\\
&(f,g)&\mapsto&\inf\{\epsi>0\mid \sup_{x\in B(x_0,1/\epsi)}|f(x)-g(x)|<\epsi\}.
\end{array}$$
Moreover if $D\subseteq X$ is a dense countable subset, then the $\QQ$-algebra generated by the functions $\{d_x\}_{x\in D}$ and the constant function $\mathbf 1$ is a dense countable subset of $\C(x)$.
\end{lem}

\begin{proof}[On the proof]
The proof of the first part is an easy exercise. The second part can be proven by applying for each $R\ge1$ the Stone-Weierstrass theorem (see \textit{e.g} \cite[p. 198]{Gil87}) to the space $X\cap\overline{B(x_0,R)}$ and the set of functions $\{d_x\}_{x\in\overline{B(x_0,R)}\cap D}$.
\end{proof}

\noindent Let $(\Omega,\A)$ be a Borel space and $\se X$ a Borel field of proper metric spaces. We fix $\se x^0\in\callo X$ and we consider the family of metrics $\se\delta=\{\om\delta\}_{\omega\in\Omega}$ on $\C(\se X)=\{\C(\om X)\}_{\omega\in\Omega}$ given by the Lemma \ref{LemmaSpaceContinuousFunctionProperSpace}.

\begin{thm}\label{ThmStrucBorelFctCont}
The set
$${\mathcal L}(\Omega,\C(\se X)):=\{\se f\in\mathcal S(\Omega,\C(\se X)) \mid\se f(\se x)\text{ is Borel for every }\se x\in\mathcal L(\Omega,\se X)\}$$
is a Borel structure on $(\C(\se X),\se\delta)$. Moreover, the subfield $\C_0(\se X)$ where
$$\C_0(\om X)=\{f\in\C(\om X)\mid f(\om x^0)=0\}$$
is Borel.
\end{thm}

\begin{proof}
This set is clearly closed under countable Borel gluings and pointwise limits, so by Lemma \ref{LemmaBorelGluingPoncutalLimit} we only have to check the points (i) and (iii) of the Definition \ref{DefiBorelField}.

\noindent Fix $\epsi>0$ and choose $\D$ a fundamental family of the Borel subfield $B(\se x^0,1/\epsi)$ (\textit{cf. }Example \ref{ExampOpenBall}). Pick $\se f,\se g\in{\mathcal L}(\Omega,\C(\se X))$. Then
$$\om\delta(\om f,\om g)\le\epsi\iff \sup_{\se x\in\D}|\om f(\om x)-\om g(\om x)|\le\epsi$$
and so ${\mathcal L}(\Omega,\C(\se X))$ is compatible with the family of metrics $\se\delta$ and the point (i) of the definition is verified.

\noindent Now observe that $\mathcal S(\Omega,\C(\se X))$ is naturally an algebra : if $\se f,\se g\in{\cal S}(\Omega,\C(\se X))$ and $\lambda\in\RR$, we can define
$$(\se f\se g)_\omega:=\om f\om g,\quad\om{(\se f+\se g)}=\om f+\om g\quad\text{et}\quad\om{(\lambda\se f)}:=\lambda\om f.$$
It's clear from its definition that the subset ${\mathcal L}(\Omega,\C(\se X))$ is a subalgebra of $\mathcal S(\Omega,\C(\se X))$. We now fix $\D=\{\se x^n\}_{n\ge1}$ a fundamental family of the Borel field $\se X$ and we define the following elements of ${\mathcal L}(\Omega,\C(\se X))$ :
$$\begin{array}{cccl} d_{\se{x^n}} : & \Omega&\longrightarrow& \mathcal{C}(X_{\point})\\
&\omega&\mapsto& \begin{array}{cccl}
                d_{x_{\omega}^n}:& X_{\omega}&\rightarrow&\mathbf{R}\\
                & x&\mapsto&d_{\omega}(x_{\omega}^n, x)\\
\end{array}\\
\end{array}\quad\text{et}\quad
\begin{array}{cccl}
\bold{1}_{\point} : & \Omega&\longrightarrow& \mathcal{C}(X_{\point})\\
&\omega&\mapsto& \begin{array}{cccl}
                \bold{1}_{\omega}:& X_{\omega}&\rightarrow&\mathbf{R}\\
                & x&\mapsto&1\\
\end{array}\\
\end{array}$$

\noindent We write $\AQ$ the countable $\QQ$-subalgebra of $\mathcal{S}(\Omega, \CXpoint)$ generated by $\{d_{\se x}\}_{\se x\in\D}\cup\{\se{\bold 1}\}$. Then $\AQ$ is contained in ${\mathcal L}(\Omega,\C(\se X))$ because this last set is an algebra that contains the generators of $\AQ$; moreover $\{\om f\mid\se f\in\AQ\}$ is dense in $\C(\om X)$ for every $\omega\in\Omega$ by Lemma \ref{LemmaSpaceContinuousFunctionProperSpace} so that the Condition \ref{DefiBorelField} (iii) is satisfied.

\noindent To prove that the subfield $\C_0(\se X)$ is Borel, it's enough to realize that if $\D$ is a fundamental family of the field $\C(\se X)$, then $\widetilde\D:=\{\se f-\se f(\se x^0)\mid\se f\in{\mathcal L}(\Omega,\C(\se X))\}$ is obviously a fundamental family of the subfield.
\end{proof}

\noindent We'll show now that the intersection behaves better in proper spaces than in complete ones (see the Proposition \ref{PropIntersectionClosedField}).

\begin{prop}\label{PropProperIntersectionClosedField}
Let $(\Omega,\A)$ be a Borel space and let $\se X$ be a Borel field of proper metric spaces. Let $\{\se F^n\}_{n\ge1}$ be a family of Borel subfield of closed sets. Then the subfield $\cap_{n\ge1}\se F^n$ is a Borel subfield.
\end{prop}

\noindent We'll need the following Lemma whose proof is straightforward.

\begin{lem}\label{LemmaDistanceProper}
Let $X$ be a proper metric space. Then the following assertions are true.
\begin{itemize}
\item[(i)] Let $F$ be a closed subset of $X$. Then for every $x\in X$ the distance $d(x,F)$ is realized.
\item[(ii)] Let $\{F_n\}_{n\ge1}$ be a decreasing sequences of closed subsets of $X$. Then for every $x\in X$ we have $d(x,\cap_{n\ge1}F_n)=\lim_{n\tend\infty}d(x,F_n)=\sup_nd(x,F_n)$ where $d(x,\emptyset)=\infty$ by convention.
\end{itemize}
\end{lem}

\begin{proof}[Proof of Proposition \ref{PropProperIntersectionClosedField}]
We'll make the proof in three steps.
\smallskip

\noindent (1) The Proposition \ref{PropDistanceBorel} - applied twice - implies that the conclusion of the theorem is true in the particular case when $\om F^{n+1}\subseteq\om F^{n}$ for every $\omega\in\Omega$ and $n\ge1$, since $\se d(\se x,\cap_{n\ge1}\se F^n)=\lim_{n\tend\infty}\se d(\se x,\se F^n)$ by Lemma \ref{LemmaDistanceProper}.
\smallskip

\noindent (2) Let $\se f\in\mathcal L(\Omega,\C(\se X))$ and $\se a,\se b\in\mathcal L(\Omega,\RR)$. Let $\se F:=\se f^{-1}([\se a,\se b])$ be the subfield defined by $\om F=\om f^{-1}([\om a,\om b])$. We'll show that $\se F$ is a Borel subfield. For every integer $n\ge1$ we set $\se U^n:=\se f^{-1}(]\se a-1/n,\se b+1/n[)$ for the subfield defined similar to $\se F$. It's a Borel subfield by Lemma \ref{LemmaSubfieldOpenSubsets}, because if $\se x\in\callo X$, then
$$\omdansom{\om x\in\om U^n}=\omdansom{\om a-1/n\le\om f(\om x)\le\om b+1/n}$$and the latter is Borel by the definition of $\mathcal L(\Omega,\C(\se X))$. By the Remark \ref{RemFieldBetweenAndUnion}, $\overline{\se U^n}$ is a Borel subfield, so that the sequence $\{\overline{\se U^n}\}_{n\ge1}$ is a decreasing sequence of Borel subfields of closed subsets such that $\cap_{n\ge1}\overline{\se U^n}=\se F$ (this equality being satisfied because every $\om f$ is continuous). Thus $\se F$ is a Borel subfield by the first step.
\smallskip

\noindent (3) We'll show now that if $\se F$ and $\se G$ are Borel subfields of closed subsets then so is $\se F\cap\se G$ (and the conclusion of the theorem will then follow by applying recursively this fact and by using step (1)). By the Proposition \ref{PropDistanceBorel} $d_{\se F},d_{\se G}\in\mathcal L(\Omega,\C(\se X))$, so that $\se F\cap\se G=\big(d_{\se F}+d_{\se G}\big)^{-1}(0)$ is a Borel subfield by step (2).
\end{proof}

\section{Borel fields of CAT(0) spaces}\label{SecBorelCAT(0)}

\subsection{Basic Properties}

\noindent First recall some notation and  terminology. A subset $C\subseteq X$ is \textit{convex} if it contains  any geodesic segment joining any two of its points. For such a closed convex subset in a complete CAT(0) space we denote by $\pi_C(x)$ the unique point which satisfies $d(x,\pi_C(x))=d(x,C):=\inf_{y\in C}d(x,y)$ \cite[II.2.4]{BH99}. This is the \textit{projection} of $x$ on $C$ and the \textit{projection map} $\pi_C : X\rightarrow C$ does not increase distances. The \textit{circumradius} of a non empty bounded set $A\subseteq X$ is $r(A):=\inf\{r>0\mid\exists x\in X, A\subseteq\overline{B}(x,r)\}$. This infimum is achieved and there exists a unique point $c_A\in X$ such that $A\subseteq\overline{B}(c_A,r(A))$ \cite[II.2.7]{BH99}. This point is called the \textit{circumcenter} of $A$. Obviously, if $\ch X$ is a field of metric spaces such that $\om X$ is a CAT(0) spaces for all $\omega\in \Omega$, we call it a \textit{field of CAT(0) spaces}. \textsc{Monod} was the first to consider such fields in \cite{Mon06}.

\begin{defi}\label{DefiBorelGeodProj}
Let $(\Omega, \A)$ be a Borel space, $(\Omega, \se X)$ be a Borel field of proper CAT(0) spaces, $\se x, \se y\in\callo{X}$ be two sections and $\ch{C}$ be a Borel subfield of non empty closed convex sets.

\noindent We define
$$\begin{array}{cccl}
\gam{x}{y}:&[0,1]&\rightarrow&\calSe X\\
&t&\mapsto&[\gam{x}{y}(t):\omega\mapsto\gamom{x}{y}(t)],
\end{array}$$
where $\gamom{x}{y}:[0,1]\rightarrow\om X$ is the unique geodesic with constant speed such that $\gamom{x}{y}(0)=\om x$ and $\gamom{x}{y}(1)=\om y$ for all $\omega\in \Omega$.

\noindent We also define $\pi_{\se C}(\se x)\in\calSe{X}$ where $\pi_{\om C}(\om x)$ is the projection of $\om x$ on $\om C$ for all $\omega\in \Omega$.
\end{defi}

\noindent We know that the field $\overline{B}(\se x, \se r)$ is Borel if $\se x\in\callo{X}$ and $\se r$ is a non negative Borel function. Since we have on one hand $\gam{x}{y}(t)=\overline B(\se x,t\dist{x}{y})\cap\overline B(\se y,(1-t)\dist{x}{y})$ for all $t\in[0,1]$ and on the other hand $\pi_{\se C}(\se x)=\se C\cap\overline B(\se x,\dist{x}{C})$, we conclude by using the Proposition \ref{PropProperIntersectionClosedField} that the sections introduced in the Definition \ref{DefiBorelGeodProj} are Borel. In the same spirit we can define a circumradius function and a circumcenter section whenever a Borel subfield of bounded sets is given.

\begin{lem}\label{LemCentresBorel}
Let $(\Omega, \A)$ be a Borel space, $\ch{X}$ be a Borel field of proper CAT(0) spaces and $\ch{B}$ be a Borel subfield of bounded sets of $\ch{X}$.
\begin{enumerate}
\item[(i)] The circumradius function $r(\se B)$ is Borel.
\item[(ii)] The section of circumcenters is Borel, \textit{i.e.} $c_{\se B}\in\mathcal{L}(\Omega,\se X)$.
\end{enumerate}
\end{lem}

\begin{proof} (i) Recall that if $B\subseteq X$ is a bounded subset in a proper CAT(0) space, then we have the equality $r(B)=\inf_{x\in X}\{\sup_{y\in B}d(x,y)\}$. So define $\om f : \om X\rightarrow\mathbf{R}$ by $\om f(\om x)=\sup_{\om y\in\om B}d(\om x, \om y)$ for all $\omega\in\Omega$. We have $\se f\in\mathcal L(\Omega,\C(\se X))$ because if $\D\subseteq\callo{X}$, $\D'\subseteq\callo{B}$ are fundamental families, then $\se f(\se x)=\sup_{\se y\in\D'}d(\se x,\se y)$ for every $\se x\in\callo X$. Consequently we deduce that the function $r(\se B)=\inf \se f=\inf_{\se x\in\D}\se f(\se x)$ is Borel.
\smallskip

\noindent (ii) Observe that $c_{\se B}=\se f^{-1}(r(\se B))$ is Borel (\textit{cf.} proof of Proposition \ref{PropProperIntersectionClosedField}, step 2).
\end{proof}

\noindent Other Borel functions appear naturally on Borel fields of CAT(0) spaces.

\noindent First recall that for CAT(0) spaces it is possible to define several notions of angles. The \textit{comparison angle} at $p$ between $x,y$ denoted by $\overline{\angle}_p(x,y)$ is the corresponding angle in a comparison triangle. This allows us to define an infinitesimal notion of angle: if $p, x, y\in X$ and $c,c':[0,b],[0,b']\rightarrow X$ are two geodesic segments such that $c(0)=c'(0)=p$ and $c(b)=x$, $c'(b')=y$, then the \textit{Alexandrov angle} at $p$ between $x$ and $y$ is defining by $\angle_p(x,y):=\limsup_{t,t'\rightarrow 0}\overline{\angle}_p(c(t), c'(t))$ where the CAT(0) hypothesis ensure the existence of this limit.

\begin{lem}\label{LemAnglesInsideBorel}
Let $(\Omega, \A)$ be a Borel space and $\ch{X}$ be a Borel field of proper CAT(0) spaces.
\begin{enumerate}
\item[(i)] If $\se x, \se y, \se p\in\callo{X}$ are such that $\om x\neq\om p\neq\om y$ for all $\omega\in\Omega$, then the  comparison angle function $\overline{\angle}_{\se p}(\se x, \se y)$ is Borel.
\item[(ii)] If we replace in (i) the comparison angle by the Alexandrov angle the function obtained is also Borel.
\end{enumerate}
\end{lem}

\begin{proof}
(i) This assertion follows directly from the law of cosines which can be used to write the angle in terms of the distances.
\smallskip

\noindent (ii) For each $n\in\mathbf{N}$ let $\Omega_n:=\{\omega\in\Omega\mid\min\{\om d(\om p, \om x), \om d(\om p, \om y)\}\ge\tfrac{1}{n}\}\in\A$ and define two sections in $\mathcal{L}(\Omega, \se X)$ by
$$\se{c^n}\mid_{\Omega_n}:=\overline B (\se p,1/n)\cap\overline B(\se x,\dist{p}{x}-1/n)\quad\text{and}\quad\se{{\widetilde c^{\,n}}}\mid_{\Omega_n}:=\overline B (\se p,1/n)\cap\overline B(\se y,\dist{p}{y}-1/n)$$
which are prolonged in an arbitrarily Borel way on $\Omega\setminus\Omega_n$. Since $\om x\neq\om p\neq\om y$ for all $\omega\in\Omega$ we have $\Omega=\cup_{n\ge1}\Omega_n$ and thus for every $\omega\in\Omega$ there exists $n_\omega\in\mathbf{N}$ such that $\om c^n=\om c(1/n)$ and $\om{\widetilde c^{\,n}}=\om{\widetilde c}(1/n)$ for all $n\ge\om n$, where $\om c:[0,\om d(\om p,\om x)]\rightarrow\om X$ (resp. $\widetilde{c}_\omega:[0,\om d(\om p,\om y)]\rightarrow\om X$) is the geodesic from $\om p$ to $\om x$ (resp. $\om y$). By \cite[II.3.1]{BH99} we have
$$\anglal{p}{x}{y}=\lim_{n\rightarrow\infty}2\arcsin\left(\frac{n}{2}\cdot\dist{c^n}{{\widetilde{c}^{\,n}}}\right)$$
and this shows that the function is Borel.
\end{proof}

\noindent \noindent Recall \cite[II.8]{BH99} that if $X$ is a proper CAT(0) space, the boundary at infinity $\partial X$ can be defined as the set of equivalence classes of geodesic rays in $X$ where two rays are equivalent (\textit{asymptotic}) if they remain at bounded distance from each other. Often we write $c(\infty)$ for the equivalence class of the geodesic ray $c$ and a typical point of $\bd X$ is written $\xi$. Fixing a base point $x_0\in X$ leads to a bijection between $\xi\in\bd X$ and the unique geodesic ray $c_{x_0,\xi}$ starting at $x_0$ and such that $c(\infty)=\xi$. This identification can be used to define the conic topology (which turns out to be independent of the choice of $x_0$) : $\xi_n\tend\xi$ if $c_{x_0,\xi_n}(t)\tend c_{x_0,\xi}(t)$ for all $t\ge0$. An other equivalent construction uses the application $i:X\rightarrow\mathcal{C}_0(X)$ defined by $x\mapsto d_x-d(x,x_0)$, where $\mathcal{C}_0(X)$ is endowed with the topology of uniform convergence on compact sets. In general for an arbitrary proper metric space this application is \textit{not} a homeomorphism onto its image. But it is if the space is geodesic \cite{Bal95}. It can be shown that $\partial X$ is homeomorphic to $\overline{i(X)}\setminus i(X)$ : to $\xi\in\bd X$ we can associate the Busemann function $b_{x_0,\xi}: X\rightarrow\mathbf{R}$, $x\mapsto\lim_{t\rightarrow\infty} d(x,c_{x_0,\xi}(t))-d(x_0, c_{x_0,\xi}(t))$. These functions satisfies
\begin{equation}\label{EquationCocycleBusemann}
b_{z,\xi}(y)=b_{x,\xi}(y)-b_{x,\xi}(z)\quad \xi\in\bd X,\ x,y,z\in X.
\end{equation}

\noindent For all $\omega\in\Omega$ we now define the application $i_{\omega} : X_\omega\longrightarrow\mathcal{C}_0(\om X)$ by setting $\om x\mapsto d_{\om x}-d(\om x, \om x^0)$ where $\se x^0\in\calSe{X}$ is a fixed section. This will enable us to deal with the Borel structure on the fields of boundaries.

\begin{thm}\label{ThmStructBord}
Let $(\Omega, \A)$ be a Borel space, $\ch{X}$ be a Borel field of proper CAT(0) spaces and $\se x^0\in\callo{X}$.
\begin{enumerate}
\item[(i)] We have $\se i\in\widetilde{\mathcal{L}}\big(\Omega,\C(\se X,\C_0(\se X))\big)$  and the subfields $\se i(\se X)$, $\overline{\se i(\se X)}$ are Borel - where $\C_0(\se X)$ is endowed with the Borel structure inherited from $(\C(\se X),\se\delta)$ (\textit{cf.} Theorem \ref{ThmStrucBorelFctCont}).
\item[(ii)] The field $\partial\se X$ is a Borel subfield of closed sets of $\overline{\se i(\se X)}$.
\end{enumerate}
\end{thm}

\begin{proof}
(i) Since $\se x^0\in\callo{X}$ we observe that $\se i(\se x)(\se y)=\dist{x}{y}-\dist{x}{x^0}$ is Borel for every $\se y\in\callo X$. So $\se i(\se x)\in\calloconto$ and $\se i$ is a morphism which is obviously continuous. Consequently $\se i(\se X)$ is a Borel subfield of $\C_0(\se X)$ as well as $\overline{\se i(\se X)}$.
\smallskip

\noindent (ii) Observe that if $X$ is a proper CAT(0) space and $x_0\in X$ is a fixed base point we have the equality $\bd X=\bigcap_{n\in\mathbf{N}}\overline{i(X)\setminus i(\overline{B(x_0,n)})}$. We use this trick to show the assertion by using the Proposition \ref{PropProperIntersectionClosedField}. Since $\se i(\se X)\setminus\se i(\overline{B(\se x^0,n)})$ is a Borel subfield of open sets of $\se i(\se X)$ for all $n\in\mathbf{N}$, we obtain that $\bd\se X=\bigcap_{n\in\mathbf{N}}\overline{\se i(\se X)\setminus\se i(\overline{B(\se x^0,n)})}$ is also Borel.
\end{proof}

\begin{rem}\label{RemStructureBord}
In particular, the Theorem \ref{ThmStructBord} describes the sections of $\mathcal{L}(\Omega', \partial\se X)$, where $\Omega'$ is the base of the subfield $\bd\se X$ which is equal to $\omdansom{\om X\text{ is unbounded}}$. By definition of the Borel structure on $(\Omega',\C_0(\se X))$ the section $\se\xi\in\mathcal{S}(\Omega', \partial\se X)$ is Borel if and only if the function $b_{\se x^0,\se\xi}(\se x)$ is Borel for every $\se x\in\mathcal L(\Omega',\se X)$. Observe that this condition doesn't depend on the choice of $\se x^0\in \mathcal L(\Omega',\se X)$ because if $\se y^0\in\mathcal L(\Omega',\se X)$ we have $b_{\se y^0,\se\xi}(\se x)=b_{\se x^0,\se\xi}(\se x)-b_{\se x^0,\se\xi}(\se y^0)$. Therefore $\se\xi$ is Borel if and only if $b_{\se y,\se\xi}(\se x)$ is Borel for every $\se x,\se y\in\mathcal L(\Omega',\se X)$.
\end{rem}

\noindent The Borel structure on the field of boundaries is such that the natural sections and functions associated are Borel.

\begin{lem}\label{LemGeodRayBorel}
Let $(\Omega, \A)$ be a Borel space, $\ch{X}$ be a Borel field of proper unbounded CAT(0) spaces and two sections $\se x\in\callo{X}$, $\se\xi\in\mathcal{L}(\Omega, \partial\se X)$. Define
$$\begin{array}{cccl}
c_{\se x, \se\xi}:&[0,\infty[&\rightarrow&\calSe X\\
&t&\mapsto&c_{\se x, \se\xi}(t):\omega\mapsto c_{\om x, \om\xi}(t).
\end{array}
$$
where $c_{\om x, \om\xi}: [0,\infty[\rightarrow\om X$ is the unique geodesic ray such that $c_{\om x, \om\xi}(0)=\om x$ and $c_{\om x, \om\xi}(\infty)=\om\xi$. Then we have $c_{\se x, \se\xi}(t)\in\callo{X}$ for every $t\in\mathbf{R}_+$.
\end{lem}

\begin{proof}
Let $\D\subseteq\callo{X}$ be a fundamental family. Since $\se i(\D)$ is also a fundamental family for the structure $\mathcal{L}(\Omega, \overline{\se i(\se X)})$ each Borel section in this set is a pointwise limit of countable Borel gluings of elements of $\se i(\D)$ by Lemma \ref{LemmaBorelStructureFundamentalFamily}. In particular there exists a sequence $\{\se x^n\}_{n\ge1}\subseteq\D$ such that $\lim_{n\rightarrow\infty}\se i(\se x^n)=\se\xi$ and we might as well suppose that $\om d(\om x^n, \om x)\ge n$ for all $\omega\in\Omega$. Consequently we can define $\gamma_{\se x, \se x^n}(t)\in\callo{X}$ at least for each $t\in[0,n]$. Now fix $t\in\mathbf{R}_+$. Since by \cite[II.8.19]{BH99} $\lim_{n\rightarrow\infty}\gamma_{\om x, \om x^n}(t)=c_{\om x, \om\xi}(t)$ for each $\omega\in\Omega$ we deduce $c_{\se x, \se\xi}(t)=\lim_{n\rightarrow\infty}\gamma_{\se x, \se x^n}(t)\in\callo{X}$.
\end{proof}

\noindent Recall that for every $\eta,\xi\in\bd X$ and $x\in X$, the Alexandrov angle between $\xi$ and $\eta$ in $x$ is defined by $\angle_x(\xi,\eta)=\angle_{x}(c_{x, \xi}(1),c_{x,\eta}(1))$ and the Tits angle between $\xi$ and $\eta$ by $\angle(\xi,\eta)=\sup_{x\in X}\angle_x(\xi,\eta)$. The Tits angle define a metric on the boundary which is called the \textit{angular metric}.

\begin{lem}\label{LemmaAngleBorel}
Let $(\Omega, \A)$ be a Borel space, $\ch{X}$ be a Borel field of proper CAT(0) spaces, a section $\se x\in\callo{X}$ and two sections $\se\xi, \se\eta\in\mathcal{L}(\Omega, \partial\se X)$.
\begin{enumerate}
\item[(i)] The Alexandrov angle function $\angle_{\se x}(\se \xi, \se \eta)$ is Borel.
\item[(iii)] The Tits angle function $\se\angle(\se\xi, \se\eta)$ is Borel.
\end{enumerate}
\end{lem}

\begin{proof}
By definition we have ${\angle}_{\se x}(\se{\xi},\se\eta)=\angle_{\se x}(c_{\se x, \se\xi}(1),c_{\se x, \se\eta}(1))$ and by \cite[II.9.8(4)]{BH99} $\se{\angle}(\se{\xi},\se\eta)=2\arcsin(\lim_{t\tend\infty} \frac{1}{2t}\cdot d(c_{\se x, \se\xi}(t),c_{\se
x, \se\eta}(t)))$. So we deduce from lemmas \ref{LemAnglesInsideBorel} and \ref{LemGeodRayBorel} that these functions are Borel.
\end{proof}

\noindent We turn now to some subfields of the field of metric spaces $(\Omega, (\partial\se X, \se\angle))$. Notice that the latter is not always a Borel field of metric spaces because the topology induced by the angular metric may be not separable. Despite this trouble we prove the following theorem.

\begin{thm}\label{ThmCircum}
Let $(\Omega, \A)$ be a Borel space, $\ch{X}$ be a Borel field of proper unbounded CAT(0) spaces and $\ch{A}$ be a Borel subfield of non empty closed sets - with respect to the conic topology - of $(\Omega, \partial\se X)$.
\begin{enumerate}
\item[(i)] The circumradius - with respect to the angular metric - function $r(\se A)$ is Borel.
\end{enumerate}
Moreover suppose that $r(\om A)<\tfrac{\pi}{2}$ for all $\omega\in\Omega$.
\begin{enumerate}
\item[(ii)] The section of circumcenters is Borel, \textit{i.e.} $c_{\se A}\in\mathcal{L}(\Omega,\partial\se X)$.
\end{enumerate}
\end{thm}

\begin{proof}
We'll make the proof of (i) in three steps.
\smallskip

\noindent (1) For a CAT(0) space $X$ and $x_0\in X$ define for each $n\in\mathbf{N}$ the function
$$\begin{array}{cccl}
\angle^n:&\bd X\times\bd X&\rightarrow&[0,\pi]\\
&(\xi,\eta)&\mapsto&\angle^n(\xi,\eta):=\displaystyle\sup_{t\in [1,n]}\overline{\angle}_{x_0}(c_{x_0,\xi}(t),c_{x_0,\eta}(t)).\end{array}$$
This increasing sequence of functions verifies $\angle(\xi,\eta)=\lim_{n\rightarrow\infty}\angle^n(\xi,\eta)=\sup_{n\ge1}\angle^n(\xi,\eta)$ by \cite[II.9.8(1)]{BH99}. If $\partial X$ is endowed with the cone topology, then $\angle^n$ is a continuous function. The argument is as follows. The function $f_n:[1,n]\times(\bd X)^2\rightarrow [0,\pi]$, $(t,\xi,\eta)\mapsto \overline{\angle}_{x_0}(c_{x_0,\xi}(t),c_{x_0,\eta}(t))$ is continuous, hence uniformly continuous. It is easy then to check that the function $(\xi, \eta)\mapsto\angle^n(\xi,\eta)=\sup_{t\in[1,n]}f^n(t,\xi,\eta)$ is continuous.
\smallskip

\noindent (2) We'll now prove that if $A\subseteq\bd X$ is a non empty closed subset, then we have the equality
\begin{equation}\label{LemExpressionRayon}
r(A)=\lim_{n\rightarrow\infty}\min_{\xi\in\partial X}\{\max_{\eta\in A}\angle^n(\xi,\eta)\}.
\end{equation}
Indeed we have
\begin{eqnarray*}
r(A)&\overset{\text{def.}}{=}&\inf_{\xi\in\bd X}\{\sup_{\eta\in A}\{\sup_{x\in X}\angle_x(\xi,\eta)\}\}=\inf_{\xi\in\bd X}\{\sup_{\eta\in A}\{\sup_{n\ge1}\angle^n(\xi,\eta)\}\}=\inf_{\xi\in\bd X}\{\sup_{n\ge1}\{\sup_{\eta\in A}\angle^n(\xi,\eta)\}\}
\end{eqnarray*}
and by (1) and compactness $r(A)=\inf_{\xi\in\bd X}\{\sup_{n\ge1}\{\max_{\eta\in A}\angle^n(\xi,\eta)\}\}$. Now consider the restriction $\angle^n\mid_{\partial X\times A}$ which is uniformly continuous by (1) and also the function $\angle^n_{\max}: \partial X\rightarrow[0,\pi]$, $\xi\mapsto\max_{\eta\in A}\{\angle^n(\xi,\eta)\}$ which is continuous. We'll use the following observation whose proof is not a very difficult exercise.
\smallskip

\noindent \textit{Observation :} Let $Y$ be a compact metrizable space and $\{g_n:Y\rightarrow[a,b]\}_{n\ge1}$ be an increasing sequence of continuous functions. Taking punctual limit gives a function $g$ which is obviously lower semi-continuous. Then $\min g=\lim_{n\tend\infty}\min g_n$. Moreover if $x_n$ is such that $g_n(x_n)=\min g_n$ then any accumulation point $x$ satisfies $g(x)=\min f$.

\noindent We can now easily deduce the formula (\ref{LemExpressionRayon}) by applying the first assertion of the observation to the sequence of functions $g_n(\xi):=\angle^n_{\max}(\xi)$, since $r(A)=\min g$.
\smallskip

\noindent (3) Finally, if $\se\xi, \se\eta\in\mathcal{L}(\Omega,\bd\se X)$, then the function $\se\angle^n(\se\xi,\se\eta)$ is Borel for every $n\ge1$ because on one hand we have by continuity
$$\se\angle^n(\se\xi,\se\eta)=\sup_{t\in[1,n]\cap\mathbf{Q}}\overline{\angle}_{\se x^0}(c_{\se x^0, \se\xi}(t), c_{\se x^0, \se\eta}(t))$$
and on the other hand since $c_{\se x^0, \se\xi}(t), c_{\se x^0, \se\eta}(t)\in\callo{X}$ the function $\overline{\angle}_{\se x^0}(c_{\se x^0,\se\xi}(t), c_{\se x^0, \se\eta}(t))$ is Borel by Lemma \ref{LemAnglesInsideBorel}. Consequently, if $\D\subseteq\mathcal{L}^0(\Omega, \bd\se X)$ and $\D'\subseteq\callo A$ are fundamental families, then we have
$$\rad(\se A)=\lim_{n\tend\infty}\min_{\se\xi\in\D}\{\max_{\se\eta\in \D'}\se\angle^n(\se\xi,\se\eta)\}$$
and this shows that the function is Borel.
\smallskip

\noindent We now undertake the proof of (ii). We'll also make this proof in three steps.
\smallskip

\noindent (1) By hypothesis for each $\omega\in\Omega$, the set $\om A$ has an unique circumcenter $c_{\om A}$ \cite[II.9.13 \& II.2.7]{BH99}. Therefore the section $c_{\se A}\in\mathcal{S}(\Omega,\partial\se X)$ is well-defined.
\smallskip

\noindent (2) Observe the following general fact. If $\ch Y$ is a Borel field of compact spaces and $\{\se f^n\}_{n\ge1}\subseteq\mathcal L(\Omega,\C(\se Y))$ is a sequence of continuous morphisms which is increasing (\textit{i.e.} for each $\omega\in\Omega$, $n\ge1$ it satisfies $\om f^{n+1}\ge\om f^n$), bounded (\textit{i.e.} for each $n\ge1$ it satisfies $\sup_{n\ge1}\om f^n<\infty$) and such that $\se f:=\lim_{n\rightarrow\infty}\se f^n$ satisfies $|\om f^{-1}(\{\min \om f\})|=1$ for all $\omega\in \Omega$, then $\se f^{-1}(\{\min \se f\})\in$ $\callo Y$. Indeed if $\D\subseteq\callo Y$ is a fundamental family we have $\min\se f^n=\inf_{\se x\in\D}\se f(\se x)\in\mathcal{L}(\Omega, \mathbf{R})$ and therefore $(\se f^n)^{-1}(\{\min \se f^n\})$ is a Borel field of closed subsets - see step 2 of the proof of Proposition \ref{PropProperIntersectionClosedField}. Consequently we can pick a Borel section $\se x^n$ in it and - by the observation made in the second step of (i) - $\se f^{-1}(\{\min \se f\})=\lim_{n\rightarrow\infty}\se x^n$ is Borel.
\smallskip

\noindent (3) Let $\D\subseteq\mathcal{L}(\Omega, \bd\se X)$ and $\D'\subseteq\callo A$ be fundamental families. We have seen that
$$\rad(\se A)=\lim_{n\tend\infty}\{\min_{\se\xi\in\D}\{\max_{\se\eta\in \D'}\se\angle^n(\se\xi,\se\eta)\}\}=\min_{\se\xi\in\D}\{\lim_{n\tend\infty}\{\max_{\se\eta\in\D'}\se\angle^n(\se\xi,\se\eta)\}\}.$$
Define $\se g^n(\se\xi):=\max_{\se\eta\in\D'}\se\angle^n(\se\xi,\se\eta)$. We have $\se g^n\in\mathcal L(\Omega,\C(\bd\se X))$ and observe that the sequence of morphisms $\{\se g^n\}_{n\ge1}$ is increasing, bounded and that
$$\se g:=\lim_{n\tend\infty}\se g^n\in\widetilde{\mathcal{L}}(\Omega, \F(\bd\se X, \RR))$$
is such that
$$\om g^{-1}(\{\min(\om g)\})=\{c_{\om A}\}\quad\text{for each }\omega\in\Omega$$
because $\min\om g=r(\om A)$ and $c_{\om A}$ is unique. Therefore by step (2) we obtain $c_{\se A}\in\callo{\bd X}$.
\end{proof}

\subsection{Limit Sets at Infinity}

The goal of this section is to associate a canonical Borel section $\se\xi\in\mathcal{L}(\Omega, \partial \se X)$ to a decreasing sequence $\{\se C^n\}_{n\ge1}$ of Borel subfields of convex, closed, non empty subsets in a field of proper CAT(0) spaces which satisfies the hypothesis of ``finite covering dimension''. The section we are looking for is obtained by considering the circumcenter of the Borel field of limit sets at infinity.

\begin{defi}
Let $X$ be a proper CAT(0) space and $\{C_n\}_{n\ge1}$ be a decreasing sequence of convex, closed, non empty subsets such that $\cap_{n\ge1} C_n=\emptyset$. Since the space is proper this assumption is equivalent to the fact that $\lim_{n\rightarrow\infty}d(x,\pi_{C_n}(x))=\infty$ for every $x\in X$. For $x\in X$, we consider
$$L:=\overline{\{\pi_{C_n}(x)\}_{n\ge1}}\cap\partial X=\{\text{accumulation points of }\{\pi_{C_n}(x)\}_{n\ge1}\}$$
where the closure is taken relatively to the conic topology on $\overline{X}$. Since the projection on a convex sets does not increase the distance, this set is independent of the choose point $x$. We call this set $L$ the limit set at infinity of the given sequence of subfields.
\end{defi}

\noindent First we show that this definition is independent of the choice of $x\in X$.

\begin{lem}\label{LemDefLimitSet}
Let $X$ be a proper CAT(0) space and $\{C_n\}_{n\ge1}$ like above. Then we have $\diam L_x\le\pi/2$ with respect to the angular metric.
\end{lem}

\begin{proof} By \cite[Prop. II.2.4 (3), p. 177]{BH99} we have
$$\angle_{\pi_{C_n}(x)}(x,\pi_{C_m}(x))\ge\pi/2$$
if $C_m\subsetneqq C_n$ and $n$ is large enough so that $x\notin C_n$. In particular we have $\overline\angle_{\pi_{C_n}(x)}(x,\pi_{C_m}(x))\ge\pi/2$ and thus
$$\overline\angle_{x}(\pi_{C_n}(x),\pi_{C_m}(x))\le\pi/2\quad\text{ if
}m>n\text{ are large enough so that }x\notin C_n\ \text{and}\ C_m\subsetneqq C_n.$$
Consider now any $\xi,\zeta\in L_x\subseteq\bd X$ as well as two subsequences
$\{\pi_{C_{n_k}}(x)\}_{k\ge1}$ and $\{\pi_{C_{m_k}}(x)\}_{k\ge1}$ such that $\pi_{C_{n_k}}(x)\tend\xi$ and
$\pi_{C_{m_k}}(x)\tend\zeta$ for $k\tend\infty$. By \cite[Lem. II.9.16, p. 286]{BH99} we conclude that
$$\angle(\xi,\zeta)\le\liminf_{k\tend\infty}\overline\angle_x(\pi_{C_{n_k}}(x),\pi_{C_{m_k}}(x))\le\pi/2$$
if $m_k>n_k$ are like above.
\end{proof}

\noindent The topological condition on $X$ needed to ensure the uniqueness of the circumcenter of a limit set at infinity is the following.

\begin{defi}
The order of a family $\mathcal{E}$ of subsets of a set $X$ is the largest integer $n$ such that the family $\mathcal{E}$ contains $n+1$ subsets with non empty intersection or $\infty$ if no such integer exists. If $X$ is a metrizable space it is possible to define the covering dimension (also called $\check C$ech-Lebesgue dimension) $\dim(X)$ by the following three steps:
\begin{enumerate}
\item[1.] $\dim(X)\le n$ if every finite open cover of $X$ has a finite open refinement of order $\le n$.
\item[2.] $\dim(X)=n$ if $\dim(X)\le n$ and the inequality $\dim(X)\le n-1$ does not hold.
\item[3.] $\dim(X)=\infty$ if the inequality $\dim(X)\le n$ does not hold for any $n$.
\end{enumerate}
We also define $\dim_C(X):=\sup\{\dim(K)\mid K\subseteq X\text{ compact}\}$ and refer to \cite[Chap. 7]{Eng89} for the properties of covering dimension and some equivalent definitions.

\end{defi}

\begin{rem}
Some authors refer to $\dim_C(X)$ like the geometric dimension of $X$ (see \textit{e.g.} \cite{CL10}). Note that a CAT(0) space $X$ such that $\dim_C(X)=0$ is a singleton and if it satisfies $\dim_C(X)=1$ it is a $\mathbf{R}$-tree.
\end{rem}

\begin{thm}[\cite{FNS06}, Thm. 1.7 \& Prop. 1.8, p. 309]\label{ThmFNS}
\quad\\[-0.4cm]

\noindent (i) If $X$ is a proper CAT(0) space, then the inequality $\dim_{C}(\partial X, \angle)\le\dim(X)-1$ holds.
\smallskip

\noindent (ii) If $Y$ is a complete CAT(1) space such that $\dim_C(Y)<\infty$ and $\diam(Y)\le \pi/2$, then there exists a constant $\delta>0$ which only depends on $\dim_C(Y)$ and such that the inequality $\rad(Y)\le\pi/2-\delta<\pi/2$ holds. In particular there exists an unique circumcenter $c_Y$ for $Y$ \cite[Prop. II.2.7, p. 179]{BH99}.
\end{thm}

\noindent Consequently the limit set at infinity of a decreasing sequence $\{C_n\}_{n\ge1}$ like above has an unique circumcenter. Indeed if $L\subseteq\bd X$ is the limit set at infinity of such a sequence we have $\diam(L)\le\pi/2$ by the Lemma \ref{LemDefLimitSet}. Since $(\bd X, \angle)$ is a complete CAT(1) space \cite[Thm. II.9.13, p. 285]{BH99}, the convex hull of $L$ is such that $\diam(\overline{\text{co}(L)})=\diam(L)\le\pi/2$ \cite[Lem. 4.1, p. 546]{LS97}. By hypothesis and \ref{ThmFNS} (i) we have $\dim_C(\bd X)<\infty$ and thus $\dim_C(\overline{\text{co}(L)})<\infty$. This allows us to apply \ref{ThmFNS} (ii) to the complete CAT(1) space $\overline{\text{co}(L)}$ to conclude that $\rad(L)\le\rad(\overline{\text{co}(L)})<\pi/2$ and that $L$ has an unique circumcenter.

\begin{prop}\label{PropLimitField}
Let $(\Omega, \A)$ be a Borel space, $\ch{X}$ be a Borel field of proper CAT(0) spaces with finite covering dimension and $\{\se C^n\}_{n\ge1}$ is a sequence of Borel subfields of convex, closed, non empty everywhere subsets which satisfies, for every $\omega\in\Omega$, $C^n_\omega\supseteq C^{n+1}_\omega$ for every $n\ge1$ and $\cap_{n\ge1}C^n_{\omega}=\emptyset$.
\begin{enumerate}
\item[(i)] The subfield $\se L$ of $\partial\se X$ - where, for every $\omega\in\Omega$, $L_{\omega}$ is the limit set at infinity of the sequence $\{C^n_\omega\}_{n\ge1}$ - is Borel.
\item[(ii)] The section $\se\xi:=c_{\se L}$ is Borel.
\end{enumerate}
\end{prop}

\begin{proof}
(i) Fix a section $\se x^0\in\callo X$. By definition we have
$$\om L=\overline{\{\pi_{\om C^n}(\om x^0)\}_{n\ge1}}\cap\bd\om X\quad\text{for every }\omega\in\Omega.$$
But $\overline{\{\pi_{\se C^n}(\se x^0)\}_{n\ge1}}$ is a Borel subfield of $\overline{X}_{\point}$ since
$\pi_{\se C^n}(\se x^0)\in\callo X$ for every $n\ge1$. Consequently since $\se{\overline X}$
is compact the intersection of this Borel subfield with $\partial\se X$ is Borel by Proposition \ref{PropProperIntersectionClosedField}. We conclude that $\se L$ is a Borel subfield of $\partial \se X$ by using the Remark \ref{RemFieldBetweenAndUnion} (5).
\smallskip

\noindent (ii) This follows directly from (i) and the fact that $\rad(\om L)<\pi/2$ for every $\omega\in\Omega$ by using the Theorem \ref{ThmCircum}.
\end{proof}

\begin{rem}
The previous results also hold for a generalized sequence $\{\se C^\alpha\}_{\alpha\in\RR}$ indexed by $\RR$ provided we had the following condition : $\om C^\beta=\cap_{\alpha<\beta}\om C^\alpha$ for every $\omega\in\Omega$ and $\beta\in\RR$. The limit set at infinity is in this case given by $L_\omega=\overline{\{\pi_{\om C^\alpha}(x)\}_{\alpha\in\RR}}\cap\bd\om X$ and the ``continuity'' condition is here to ensure that if $D$ is a dense subset of $\RR$, then
$$L_\omega=\overline{\{\pi_{\om C^\alpha}(x)\}_{\alpha\in D}}.$$
This is used to prove that $\se L$ is a Borel subfield.
\end{rem}

\subsection{Adams--Ballmann Decomposition}

\noindent We now turn our attention to the Adams--Ballmann decomposition of a proper CAT(0) space. First we recall the following key definition.

\begin{defi}
Let $X$ be a proper CAT(0) space. A point $\xi\in\bd X$ is called a flat point if the associate Busemann function $b_\xi$ is an affine function. Remark that the set of flat points - denoted by $F$ - is $\text{Isom}(X)$-invariant.
\end{defi}

\noindent The boundary of a product $X\times Y$ is isometric (when endowed with the angular metric) to the spherical join of the boundaries, \textit{i.e.} $\bd X\ast\bd Y$ (see \cite[I.5.13]{BH99} for the definition of the spherical join of two metric spaces and \cite[II.9.11]{BH99} for the proof of the result). The following theorem states the existence of what we shall call the Adams--Ballmann decomposition of a CAT(0) space.

\begin{thm}[\cite{AB98}]\label{ThmDecompositionAB} Let $X$ be a proper CAT(0) space. Then there exists a real Hilbert space $E$, a complete CAT(0) space $Y$ and an isometric map $i:X\tend Z=Y\times E$ such that
\begin{itemize}
\item[(i)] $i(F)=\bd E\cap\bd(i(X))\subseteq\bd Y\ast\bd E\simeq\bd Z$ and the set of directions $\{v(i(\xi))\mid\xi\in F\}$ generates $H$ as a real Hilbert space,
\item[(ii)] the set $Y':=\pi_Y(i(X))$ is convex and dense in $Y$,
\item[(iii)] any isometry $\gamma:X\tend X$ extends uniquely in $\widetilde\gamma:Z\tend Z$ and $\widetilde\gamma=(\widetilde\gamma_Y,\widetilde\gamma_E)$.
\end{itemize}
\end{thm}
\noindent It follows from this theorem that the angular and the conic topology on $F$ coincide, that the geometry on $F$ is spherical and that $F$ is closed and $\pi$-convex in $\bd X$. In order to adapt this result in the context of Borel fields of proper CAT(0) spaces we have to observe the following.

\begin{rems}\label{RemsDecompositionAB}\quad
\smallskip

\noindent (1) Let $X$ be a proper \cat\ space. If $D$ is a dense subset of $F$, then $E$ is generated by $\{v(i(\xi))\mid\xi\in D\}$. If $F=\emptyset$, then the decomposition is trivial with $E=\{*\}$ and $Y=C$.
\smallskip

\noindent (2) A careful analysis of the proof of the Theorem \ref{ThmDecompositionAB} shows that one can construct the decomposition such that the origin of $E$ is $\pi_E(i(x_0))$ where $x_0\in X$ is any chosen point.
\end{rems}

\begin{lem}\label{LemFlatPointsBorelSubfield}
Let $(\Omega, \A)$ be a Borel space and $(\Omega, \se X)$ be a Borel field of proper CAT(0) spaces. Then the subfield $\se F$ of $\bd\se X$ - defined by $\om F$ is the set of flat points of $\bd\om X$  for every $\omega\in\Omega$ - is a Borel subfield of closed subsets.
\end{lem}

\begin{proof} We start by considering a proper \cat\ space $X$ and $x_0$ a base point of $X$. For every positive integer $R$ we introduce the function
$$\begin{array}{cccl}
\Delta^R:&\C(X)&\rightarrow&\RR\\
&f&\mapsto&\Delta^R(f):=\sup_{z,z'\in B(x_0,R)}\sup_{t\in[0,1]}|f(\gamma_{z,z'}(t))-(1-t)f(z)-tf(z')|
\end{array}$$
where we recall that $\gamma_{z, z'}$ is the geodesic from $z$ to $z'$. It's straightforward to check that for every positive integer $R$ the function $\Delta^R$ is continuous - when $\C(X)$ is endowed with the uniform convergence on compact sets. Note that if $D\subseteq X$ is a dense subset and $f\in\mathcal{C}(X)$, we will obtain the same value for $\Delta^R(f)$ by taking the supremum on $B(x_0,R)\cap D$ and $[0,1]\cap\mathbf{Q}$ because in a CAT(0) space a geodesic varies continuously with its endpoints (see \cite[II.1.4]{BH99}). We'll use the functions $\Delta^R$ - that measure the lack of affinity of functions in $\C(X)$ on the balls $B(x_0,R)$ - to show that $(\Omega,\se F)$ is a Borel subfield. We fix $\se x^0\in\callo X$ and we define $\se{\overline{\Delta}}\!{}^R\in\mathcal S(\Omega,\C(\bd\se X))$ by
$$\begin{array}{cccl}
\om{\overline\Delta}\!{}^R :&\bd\om X&\tend&\RR\\
&\xi&\mapsto&\Delta^R(b_{\xi,\om x^0})
\end{array}$$
Indeed, $\om{\overline\Delta}\!{}^R$ is continuous because it is the restriction of $\om\Delta^R$ on $\bd\om X$. By definition of $\se F$ we have that $\se F=\cap_{R\ge1}(\se{\overline \Delta}\!{}^R)^{-1}(\{0\})$ so that by the Proposition \ref{PropIntersectionClosedField} it remains to show that $(\se{\overline \Delta}\!{}^R)^{-1}(\{0\})$ is a Borel subfield . By the second step of the proof of this same proposition, it's enough to show that $$\se{\overline\Delta}\!{}^R\in{\cal L}(\Omega,\C(\bd\se X)).$$
So we only have to check that $\se{\overline\Delta}\!{}^R(\se\xi)$ is a Borel function whenever $\se\xi\in{\cal  L}(\Omega,\bd\se X)$.

\smallskip
\noindent For every $R\ge1$, we pick a fundamental family $\D^R$  of the Borel subfield $B(\se x^0,R)$ of $\se X$. Remark that
$$\se{\overline\Delta}\!{}^R(\se\xi)=\underset{\se z,\se z'\in \D^R}{\sup}\underset{t\in[0,1]\cap\mathbf{Q}}{\sup}|b_{\se\xi,\se x^0}(\gamma_{\se z,\se z'}(t))-(1-t)b_{\se\xi,\se x^0}(\se z)-tb_{\se\xi,\se x^0}(\se z')|$$
which is therefore Borel because the evaluation $b_{\se\xi,\se x^0}$ on Borel sections of $\se X$ is Borel by definition and since $\gamma_{\se z,\se z'}(t)\in\callo X$ for every $t\in[0,1]$ (see the comment after the definition \ref{DefiBorelGeodProj}).
\end{proof}

\begin{prop}\label{PropDecompositionABborelien}
Let $(\Omega, \A)$ be a Borel space and $(\Omega, \se X)$ be a Borel field of proper CAT(0) spaces. For every $\omega\in\Omega$, we consider the Adams--Ballmann decomposition $\om i:\om X\tend\om Z=\om Y\times\om E$. Then there exists Borel structures $\callo Y$ and $\callo E$ on the fields $(\Omega,\se Y)$ and $(\Omega,\se E)$ such that $\se i:(\Omega,\se X)\tend(\Omega,\se Z)$ is an isometric morphism where the Borel structure on $(\Omega,\se Z)$ is given by\footnote{It can be easily seen to be a Borel structure on $(\Omega,\se Z)$.} $\callo Y\times\callo E$.
\end{prop}

\begin{proof}
We start by defining the Borel structure on $\se E$. Fix $\se x^0\in\callo X$. By the Remark \ref{RemsDecompositionAB} (2), we can choose the decomposition such that the origin of $\om E$ is $\pi_{\om E}\big(\om i(\om x^0)\big)$ for every $\omega\in\Omega$. We pick $\D:=\{\se\xi^n\}_{n\ge1}\subseteq\mathcal L(\Omega,\bd\se X)$ a fundamental family of the Borel subfield $\se F$. By Remark \ref{RemsDecompositionAB} (1) the sets $\{v(\om i(\om \xi^n))\}_{n\ge1}$ are total in $\om E$ for every $\omega\in\Omega$. Moreover, if we denote by $\langle\cdot,\cdot\om{\rangle}$ the scalar product on $\om E$, we have that the map $\Omega\rightarrow\RR,\omega\mapsto\langle v(\om i(\om xi^n)),v(\om i(\om\xi^m))\om\rangle=\cos(\angle(\om\xi^n,\om\xi^m))$ is Borel since $\angle(\se\xi,\se\eta)$ is Borel for every $\se\xi,\se\eta\in\mathcal L(\Omega,\se F)\subseteq\mathcal L(\Omega,\bd\se X)$ (\textit{cf}. Lemma \ref{LemmaAngleBorel}). So the family $\{\se\xi^n\}_{n\ge1}\subseteq\calSe E$ is a fundamental family in the sense of Dixmier \cite[p. 145]{Dix69} and so it generates a Borel structure\footnote{It is a structure in a Hilbert sense, but it's easy to see that it's a particular case of a Borel structure on a field of metric spaces.}
$$\callo E:=\{\se e\in\calSe H\mid\langle\se e,v(\se i(\se\xi^n)) \se\rangle\text{ is a Borel function }\forall\ n\ge1\}.$$
We claim that $\pi_{\se E}(\se i(\se x))\in\callo E$ for every $\se x\in\callo X$. We know that $\om i(\om\xi^n)\in\bd\om E\subseteq\bd\om Y\ast\bd\om E=\bd\om Z$ for every $n\ge1$ and $\omega\in\Omega$. Thus, by the classical descriptions of Busemann functions in a product and in a Hilbert space (keep in mind here that $\om i(\om x^0)$ is the origin of the Hilbert space $\om E$), we have that
$$\langle\pi_{\se E}(\se i(\se x)),v(\se i(\se\xi^n))\se\rangle=-b_{\se i(\se x^0),\se i(\se\xi^n)}(\se i(\se x))=
-b_{\se x^0,\se\xi^n}(\se x).$$
The last function being Borel (\textit{cf.} Remark \ref{RemStructureBord}) we've therefore proven the claim.
\smallskip

\noindent Now we can deal with the structure on $\se Y$. Let $\D=\{\se x^n\}_{n\ge1}$ be a fundamental family of the field $\se X$. Then $\{\pi_{\om Y}(\om i(\om x^n))\}_{n\ge1}$ is dense in $\om Y$ for every $\omega\in\Omega$. Moreover $$d_{\se Y}(\pi_{\se Y}(\se i(\se x^n)),\pi_{\se Y}(\se i(\se x^m)))=\sqrt{d_{\se X}(\se x^n,\se x^m)^2-d_{\se Y}(\pi_{\se Y}(\se i(\se x^n)),\pi_{\se Y}(\se i(\se x^m)))^2}$$
is a Borel function for every $n,m\ge1$. By the Example \ref{ExampBorelField} (5) the family $\{\pi_{\se Y}(\se i(\se x^n))\}_{n\ge1}$ defines a Borel structure $\callo Y$ on $(\Omega,\se Y)$. As before we can easily show that if $\se x\in\callo X$, then $d_{\se Y}(\pi_{\se Y}(\se i(\se x^n)),\pi_{\se Y}(\se i(\se x)))$ is Borel for every $n\ge1$, \textit{i.e.} $\pi_{\se Y}(\se i(\se x))\in\callo Y$.

\noindent Therefore there exists Borel structures $\callo Y$ and $\callo Z$ such that $\se i(x)\in\callo Y\times\callo Z$ for every $\se x\in\callo X$, \textit{i.e.} $\se i$ is a morphism.
\end{proof}

\noindent There is two important subsets of $F$ that are used in the proof of Theorem \ref{ThmAB} : the subset $A:=\{\xi\in F\mid-\xi\in F\}$ - which is well defined since the geometry of $F$ is spherical - and $P:=\{\xi\in F\mid\angle(\xi,A)=\frac{\pi}{2}\}$. Observe that these subsets are closed and $\pi$-convex, that if we decompose $X$ with respect to $A$ then we have $X\simeq Y\times\RR^n$ and that $P=\emptyset$ if and only if $A=F$.

\begin{lem}\label{LemseP&seA}
Let $(\Omega, \A)$ be a Borel space and $(\Omega, \se X)$ be a Borel field of proper CAT(0) spaces. Then the subfields $\se A$ and $\se P$ of $\se F$ are Borel.
\end{lem}

\begin{proof}
Recall that $\om F\subseteq\bd\om E$ and observe that $\bd\om E$ can be interpreted (topologically) as the unit sphere of $\se E$. From Proposition \ref{PropDecompositionABborelien} we have that $\se F$ is a Borel subfield of $\bd\se E$. Since $\om E$ is a Hilbert space, we can consider $-\om F\subseteq \bd\om E\subseteq \om E$ and it's easy to get convinced that $\om F\cap(-\om F)=\om A$. As $-\se F$ is obviously a Borel subfield of compact subsets of $\se E$, then $\se A$ is a Borel subfield of $\se F$ by Proposition \ref{PropProperIntersectionClosedField}. Thus $\se\angle(\cdot,\se A)\in {\cal L}(\Omega,\C(\se F))$. Therefore - by the second step of the proof of this same proposition - $\se P=\se\angle(\cdot,\se A)^{-1}(\{\pi/2\})$ is also a Borel field of closed subsets of $\se F$.
\end{proof}

\section{Actions of Equivalence Relations}

\subsection{Definition}

\noindent In the sequel $[\R]$ denotes the full group of $\R$, \textit{i.e.} the group of Borel automorphisms of $\Omega$ whose graphs are contained in $\R$.

\begin{defi}\label{DefiAction}
Let $(\Omega, \A)$ be a standard Borel space, $\ch{X}$ be a Borel field of metric spaces and $\R\subseteq\Omega^2$ be a Borel equivalence relation. An action of $\R$ on $\ch{X}$ is given by a family of bijective maps indexed by $\R$ denoted by $\{\alpha(\omega, \omega') : \om X\rightarrow\omp X\}_{(\omega, \omega')\in\R}$ with the following conditions.
\begin{enumerate}
\item[(i)] [Cocycle rule] For every $(\omega, \omega'), (\omega', \omega'')\in\R$, $\alpha(\omega', \omega'')\circ\alpha(\omega, \omega')=\alpha(\omega, \omega'')$.
\end{enumerate}
Observe that condition (i) implies the existence of a natural action of $[\R]$ on $\calSe X$ : if $g\in[\R]$ and $\se x\in\calSe{X}$, we can define a new section $g\se x$ by $(g\se x\om)=\alpha(g^{-1}\omega,\omega) x_{g^{-1}\omega}$. As we are in the Borel context, we'll also require:
\begin{enumerate}
\item[(ii)]  The set $\callo X\subseteq\calSe X$ is invariant under the action of $[\R]$.
\end{enumerate}

\noindent We denote such an action by $\alpha : \R\curvearrowleft\ch{X}$.

\noindent If moreover for all $(\omega, \omega')\in\R$ the map $\alpha(\omega, \omega')$ is continuous (resp. isometric or linear) we say that $\R$ acts by homeomorphisms (resp. by isometries or linearly).
\end{defi}

\subsection{Basic Properties}

\noindent If $\R$ acts by homeomorphisms it is straightforward to see by using the Lemma \ref{LemmaBorelStructureFundamentalFamily} that the Condition (ii) is equivalent to $g(\D)\subseteq\callo{X}$ for all $g\in[\R]$ for any fundamental family $\D\subseteq\callo{X}$. By using classical technics of decompositions and gluings, it's also possible to prove that it's enough to check condition (ii) for every element of a countable group $G\subseteq[\R]$ such that $\R=\R_G$. Observe also that $[\R]$ acts on the set of subfields of $\ch{X}$ in total analogy with its action on the sections.

\begin{prop}\label{PropActionGSecSubfield}
Let $(\Omega, \A, \mu)$ be a standard probability space, $\ch{X}$ be a Borel field of metric spaces and $\R\subseteq\Omega^2$ be a Borel equivalence relation which acts on $\ch{X}$ by homeomorphisms. Then $[\R]$ acts on the set of Borel subfields of $\ch{X}$. Moreover if $\mu$ is quasi-invariant under $\R$, then $[\R]$ acts on $L(\Omega, \se X)$ and more generally it acts on the equivalence classes of Borel subfields of $\ch{X}$.
\end{prop}

\begin{proof}
Let $\se{A}$ be a Borel field, $\Omega':=\{\omega\in\Omega\mid\om A\nonvide\}$ its base and $g\in[\R]$. To show that the field $g\se A$ defined by $(g\se A)_\omega:=\alpha(g^{-1}\omega, \omega)A_{g^{-1}\omega}$ is Borel, we observe that $\{\omega\in\Omega\mid (g\se A)_{\omega}\neq\emptyset\}=g\Omega'\in\A$ and that if $\D\subseteq\calloom{\Omega'}{A}$ is a fundamental family of $\se A$ then $g(\D)\subseteq\mathcal L(g\Omega',g\se A)$ satisfies \ref{DefiBorelSubfield} (ii). Under the hypothesis of quasi-invariance of $\mu$, if $\se A=_{\mu\text{-a.e.}}\se B$ then $g\se A=_{\mu\text{-a.e.}}g\se B$ since the set $\{\omega\in\Omega\mid (g\se A)_{\omega}=(g\se B)_{\omega}\}=g\{\omega\in\Omega\mid \om A=\om B\}$ is of measure null.
\end{proof}

\begin{defi}\label{DefiInvariance}
Under the hypothesis of the Proposition \ref{PropActionGSecSubfield}, a section $\se x\in\callo{X}$ is called $\R$-invariant (or simply invariant whenever the relation can clearly be identified) if $\alpha(\omega, \omega')\om x=\omp x$ for all $(\omega, \omega')\in\R$. We say that a section $\se x$ is almost $\R$-invariant if it exists a Borel set $A\in\A$ with $\mu(A)=1$ such that the equality holds for all $(\omega, \omega')\in\R\cap A^2$. Mutatis mutandis we define $\R$-invariant and almost $\R$-invariant Borel subfields. We'll also use the terminology (almost) $\alpha$-invariant whenever it is necessary to be more precise.
\end{defi}

\begin{rem}\label{RemInvariance}
It's straightforward to check that a Borel section $\se x$ (or a Borel subfield) is (almost) $\R$-invariant if and only if it is (almost) $[\R]$-invariant. Therefore the almost invariance of a Borel section $\se x$ (or of a Borel subfield) is equivalent to the invariance of its class $[\se x]\in L(\Omega,\se X)$. The existence of a countable group generating $\R$ allows us to always assume that the set $A$ of the definition \ref{DefiInvariance} is invariant. An the other hand if $\ch Y$ is an invariant Borel subfield, then the action on $\ch X$ induced an action on $\ch Y$. Putting it all together, it shows that - in the measure context - we can always assume without losing generality that an almost invariant Borel subfield is invariant.
\end{rem}

\noindent We now show that an action of $\R$ on a Borel field of proper metrics spaces $\ch{X}$ naturally gives rise to an action on the previously constructed Borel field $(\Omega, \mathcal{C}(\se X))$.

\begin{lem}\label{LemActionC}
Let $(\Omega, \A)$ be a standard Borel space, $\ch{X}$ be a Borel field of proper metric spaces and $\R\subseteq\Omega^2$ be a Borel equivalence relation. Suppose that an action $\alpha : \R\curvearrowleft\ch{X}$ by homeomorphisms is given. Then there exists an induced action $\widetilde{\alpha} : \R\curvearrowleft(\Omega, \mathcal{C}(\se X))$ by linear homeomorphisms.
\end{lem}

\begin{proof}
\noindent The natural way to define the action is to write for $(\omega, \omega')\in\R$
$$\begin{array}{cccl}\widetilde\alpha(\omega,\omega'):&\C(\om X)&\tend&\C(\omp X)\\
&\om f&\mapsto&\om f\circ\alpha(\omega',\omega). \end{array}$$
It is clear that $\widetilde\alpha(\omega,\omega')$ is a homeomorphism (with respect to the topology of uniform convergence on compact sets) because a homeomorphism between Hausdorff spaces preserves compact sets. The cocycle rule is obvious so it only remains to check the Condition \ref{DefiAction} (ii), \textit{i.e.} that if $\se f\in{\mathcal L}(\Omega,\C(\se X))$ and $g\in[\R]$, then $g\se f\in{\mathcal L}(\Omega,\C(\se X))$. To do so, we fix $\se x\in\callo{X}$ and we observe that
$$(g\se f)_\omega(\om x)=f_{g^{-1}\omega}\big(\alpha(\omega,g^{-1}\omega)(\om x)\big)=f_{g^{-1}\omega}\big(g^{-1}(\se x)_{g^{-1}\omega}\big).$$
Consequently we have $g\se f(\se x)=(\se f(g^{-1}\se x))\circ g^{-1}$ and this shows that the evaluation $(g\se f)(\se x)$ is Borel because $g^{-1}\se x\in\callo X$ (since $\alpha$ is an action) and $\se f(g^{-1}\se x)$ is Borel  by definition of the Borel structure of the Borel field $(\Omega, \mathcal{C}(\se X))$. So we can conclude that $g\se f\in{\mathcal L}(\Omega,\C(\se X))$.
\end{proof}

\noindent Moreover if the field is a field of CAT(0) spaces  then the action extends to $(\Omega,\bd\se X)$.

\begin{lem}\label{LemActionBd}
Let $(\Omega, \A)$ be a Borel space, $\ch{X}$ be a Borel field of proper unbounded CAT(0) spaces, $\R\subseteq\Omega^2$ be a Borel equivalence relation and an action $\alpha : \R\curvearrowleft\ch{X}$ by isometries. Then there exists an induced action $\widetilde{\alpha} : \R\curvearrowleft(\Omega, \overline{X}_{\point})$ by homeomorphisms and $(\Omega,\partial\se X)$ is invariant.
\end{lem}

\begin{proof}
Let $X_1$ and $X_2$ be two proper unbounded CAT(0) spaces and $\gamma : X_1\rightarrow X_2$ an isometry. If we think to the boundary as the quotient of the geodesic rays then the extension $\widetilde{\gamma} :\overline{X}_1\rightarrow\overline{X}_2$ is purely geometric and is a homeomorphism such that $\widetilde\gamma(\bd X_1)=\bd X_2$ \cite[II.8.9]{BH99}. As we used the notion of Busemann functions to define the Borel structure of the field of boundaries, we need to transpose the situation to this context. If $x_i\in X_i$ are base points of $X_i$ for $i=1,2$ then the map
\begin{equation}
\begin{array}{cccl}
\widetilde\gamma_0:&\C_0(X_1)&\tend&\C_0(X_2)\\
&f&\mapsto&\widetilde\gamma_0(f):y\mapsto
f(\gamma^{-1}y)-f(\gamma^{-1}x_2)
\end{array}
\end{equation}
is a homeomorphism and it is such that the diagram
$$\begin{diagram}
\node{\C_0(X_1)}\arrow{e,t,2}{\widetilde\gamma_0}\node{\C_0(X_2)}\\
\node{X_1}\arrow{e,t,2}{\gamma}\arrow{n,t,2,J}{i_1}\node{X_2}\arrow{n,t,2,J}{i_2}
\end{diagram}
$$
commutes. It's easy to check that the two extensions coincide and that if $\xi\in\partial X_1$ then  $\widetilde{\gamma}_0(b_{x_1,\xi})=b_{x_2,\widetilde{\gamma}(\xi)}$.
\smallskip

\noindent Now let's turn to the case of fields. Given a fixed section $\se x^0\in\callo{X}$, we use (i) to define for each $(\omega, \omega')\in\R$
$$\begin{array}{cccl}
\widetilde\alpha(\omega,\omega'):&\C_0(\om X)&\tend&\C_0(\omp X)\\
&f&\mapsto&\alpha(\omega,\omega')(f):x\mapsto
f(\alpha(\omega',\omega)x)-f(\alpha(\omega',\omega)\omp x^0).
\end{array}$$
This formula defines an action by homeomorphisms. We proceed as in the proof of Lemma \ref{LemActionC} : the verification of the cocycle rule is straightforward, and an easy computation shows that for every $\se x\in\callo{X}$, $g\in[\R]$ and $\se f\in\mathcal{L}(\Omega, \mathcal{C}_0(\se X))$ we have
$$(g\se f)(\se x)=\big(\se f(g^{-1}\se x)-\se f(g^{-1}\se x^0)\big)\circ g^{-1}.$$
Thus $g\se f$ is Borel.
\end{proof}

\subsection{Amenability}

In this section we define the amenability for an equivalence relation in terms of actions on Borel fields of Banach spaces and we also show that our definition is equivalent with the one given originally by \textsc{Zimmer}.

\begin{defi}\label{DefiFieldBanachSep}
Let $(\Omega,{\cal A})$ be a Borel space and $\ch B$ be a field of Banach spaces on $\Omega$. Such a field is called Borel if it exists a Borel structure on $\ch B$ which is defined in the same way as in the Definition \ref{DefiBorelField} with the additional assumption that $\callo B\subseteq\calSe B$ is a vector space\footnote{This definition can be easily seen to be equivalent to the ones given in \cite[I, p. 77]{FD88} or \cite[p. 177]{AR00}, essentially by using the Lemma \ref{LemmaBorelGluingPoncutalLimit}.}.
\end{defi}

\noindent We can consider for each $\omega\in\Omega$ the topological dual $B^\ast_{\omega}$ which is not separable in general. Thus the field $(\Omega, B^\ast_{\point})$ has no chance to satisfy the Definition \ref{DefiFieldBanachSep} but we still have the following result.

\begin{lem}\label{LemChampBanachDual}
Let $(\Omega,{\cal A})$ be a Borel space and $\ch{B}$ be a Borel field of Banach spaces. Define
$$\widetilde{\mathcal L}(\Omega,\se B^{\ast}):=\{\se{\varphi}\in \mathcal{S}(\Omega, B_{\point}^{\ast})\,|\,\omega\mapsto\langle\varphi_{\omega}, x_{\omega}\rangle:=\varphi_{\omega}(x_{\omega})\ \text{is Borel for every}\ \se{x}\in\callo{B}\}.$$
Then $\widetilde{\mathcal L}(\Omega,\se B^{\ast})$ is a vector space which satisfies the following properties.
\begin{enumerate}
\item[(i)] For every $\se \varphi\in\widetilde{\mathcal L}(\Omega,\se B^{\ast})$ the function $\omega\mapsto\|\om \varphi\|_\omega$ is Borel.
\item[(ii)] The space $\widetilde{\mathcal L}(\Omega,\se B^{\ast})$ is closed under pointwise limits and Borel gluings.
\end{enumerate}
Moreover if $\R\subseteq\Omega^2$ is an equivalence relation and $\alpha : \R\curvearrowleft (\Omega, \se B)$ is a linear isometric action, then the fiberwise adjoint maps given by $\alpha_\ast(\omega, \omega'):=(\alpha(\omega', \omega))^\ast : B^\ast_{\omega}\rightarrow B^\ast_{\omega'}$ satisfy
\begin{enumerate}
\item[(iii)] the cocycle rule $\alpha_\ast(\omega', \omega'')\circ\alpha_\ast(\omega, \omega')=\alpha_\ast(\omega, \omega'')$ for every $(\omega, \omega'), (\omega', \omega'')\in\R$,
\item[(iv)] $g(\widetilde{\mathcal L}(\Omega,\se B^{\ast}))\subseteq \widetilde{\mathcal L}(\Omega,\se B^{\ast})$ for every $g\in[\R]$.
\end{enumerate}
Thus $\alpha_\ast$ would be an action in the sense of the Definition \ref{DefiAction} if the field was a Borel field of metric spaces.
\end{lem}

\begin{proof}
The properties (i) and (ii) are proved in \cite[A.3.6 \& A.3.7, p. 179]{AR00}, the property (iii) is obvious and the property (iv) can be proved exactly as in Lemma \ref{LemActionC}.
\end{proof}

\begin{defi}
Let $(\Omega,{\cal A},\mu)$ be a standard probability space and $\ch{B}$ be a Borel field of Banach spaces.
\begin{enumerate}
\item[(i)] Define
$$\callun{B}:=\{\se{x}\in\callo{B}\,|\,\|\se{x}\|_1:=\textstyle\int_{\Omega}\|x_{\omega}\|_{\omega}d\mu(\omega)<\infty\}$$
and $\lun{B}:=\callun{B}/=_{\mu\text{-a.e.}}$. Then $\lun{B}$ endowed with the norm $\|\;\|_1$ is a separable Banach space.
\item[(ii)] Define
$$\widetilde{\mathcal L}^\infty(\Omega,\se B^{\ast}):=\{\se{\varphi}\in\widetilde{\mathcal L}(\Omega,\se B^\ast)\,|\, \|\se{\ph}\|_{\point}:\omega\mapsto\|\ph_{\omega}\|_{\omega}\in {\mathcal L}^{\infty}(\Omega, \mathbf{R})\}.$$
and $\widetilde{L}^\infty(\Omega,\se{B^\ast}):=\widetilde{\mathcal L}^\infty(\Omega,\se B^{\ast})/=_{\mu\text{-a.e.}}$. Then $\widetilde{L}^\infty(\Omega,\se B^{\ast})$ is a Banach space when it is endowed with the norm $\|\;\|_{\infty}$ where $\|\se\ph\|_\infty$ is the $\infty$-norm of the function $\|\se\ph\|_{\point}$.
\end{enumerate}
For a detailed proof of this two assertions - which are close to the ones of the trivial field case - we refer to \cite{And10} or \cite{Hen10}.
\end{defi}

\noindent In this context, the following result holds.

\begin{prop}[\cite{AR00}, Prop. A.3.9, pp. 179-180]\label{PropLinftyDualLun}
Let $(\Omega,{\cal A}, \mu)$ be a standard probability space and $\ch{B}$ be a Borel field of Banach spaces. Then there exists an isometric isomorphism between $\widetilde L^\infty(\Omega,\se B^\ast)$ and $\big(\lun{B}\big)^{\ast}$ given by
$$\begin{array}{cccccl}
\widetilde L^\infty(\Omega,\se B^\ast)&\longrightarrow&\big(\lun{B}\big)^{\ast}&&&\\
\text{}[{\se \varphi}]&\mapsto&\langle[\se{\varphi}],\cdot \rangle&:\ \ \lun{B}&\rightarrow&\mathbf{R}\\
&&&[\se{x}]&\mapsto&\langle[\se{\varphi}],
[\se{x}]\rangle=\int_{\Omega}\langle \varphi_{\omega},
x_{\omega}\rangle d\mu(\omega).
\end{array}$$
\end{prop}

\noindent Whenever $B$ is a Banach space we use $B_{\le1}$ (resp. $B_{=1}$) to denote the closed ball (resp. sphere) of radius 1. If $A$ is a subset of $B^*$ then $\overline{A}^{\, w\ast}$ is the closure of $A$ with respect to the weak-$\ast$ topology.

\begin{defi}[\cite{AR00}, Def. 4.2.1, p. 97]\label{DefiChampcc}
Let $(\Omega,{\cal A}, \mu)$ be a standard probability space and $\ch{B}$ be a Borel field of Banach spaces.
We say that $\ch{C}$ is a Borel subfield of convex, weakly-$\ast$ compact subsets of the field of closed balls of radius 1 in the duals $(\Omega, B^{\ast}_{\point, \le1})$ if it exists a family of sections $\{\se{\varphi^n}\}_{n\ge1}\subseteq\big(\widetilde L^\infty(\Omega,\se B^\ast)\big)_{\le1}$ such that
$$C_{\omega}=\overline{\text{co}(\{\varphi_{\omega}^n\}_{n\ge1})}^{\, w\ast}\ \text{for }\mu\text{-almost every }\omega\in\Omega$$
where the closure of the convex hull is taken relatively to the weak-$\ast$ topology. Note that the set $\widetilde{L}(\Omega, \se C):=\{[\se\ph]\in \widetilde L^{\infty}(\Omega, \se B^\ast)\mid\om\ph\in\om C\text{ for } \mu\text{-almost every }\omega\in\Omega\}$ is a convex weak-$\ast$ closed subset of the unit ball $(\widetilde{L}^{\infty}(\Omega, \se B^\ast))_{\le1}$ \cite[Prop. 4.2.2, p. 97]{AR00}.
\end{defi}

\begin{rem}
It can be shown \cite{And10} that there exists $\Omega'\in\A$ of full measure and a family of metrics $\{d^\ast_{\omega}\}_{\omega\in\Omega'}$ such that $\widetilde{\mathcal L}^\infty(\Omega',\se B^\ast)$ is a Borel structure on the field $(\Omega', (B^\ast_{\point, \le 1}, d^\ast_{\point}))$. A Borel subfield of convex, weakly-$\ast$ compact subsets is then a Borel subfield in the sense of the Definition \ref{DefiBorelSubfield}.
\end{rem}

\noindent Historically \textsc{Zimmer} was the first to introduce a notion of amenability for an equivalence relation. His definition can be formulated as a particular case of the following definition - it corresponds to the case of a trivial field\footnote{To show that \textsc{Zimmer}'s definition corresponds to case of a trivial field of Banach spaces, the only thing to check is the equivalence between an action of an equivalence relation on a trivial field $(\Omega,B)$ and a Borel cocycle from $\R$ to $\text{Isom}(B)$. This can be easily done by using Corollary 1.2 of \cite{Zim78}.}.

\begin{defi}\label{DefiMoyChamps}
Let $(\Omega, \mathcal{A},\mu)$ be a standard probability space and $\R\subseteq\Omega^2$ an equivalence relation which preserves the class of $\mu$.

\noindent We say that $\R$ is amenable if for every action $\alpha$ of $\R$ by linear isometries on a Borel field of Banach spaces $\ch{B}$ and almost $\alpha_*$-invariant Borel subfield of convex, weakly-$\ast$ compact subsets
$(\Omega, \se C)$ of $(\Omega, B^\ast_{\point, \le 1})$, there exists a Borel section $\se\ph\in\widetilde{\mathcal{L}}(\Omega, \se B^\ast)$ which is almost $\R$-invariant and such that $\om\varphi\in \om C$ for almost every $\omega\in \Omega$.
\end{defi}

\noindent We'll now prove that hyperfinitness $\mu$-almost everywhere implies amenability in the sense of the Definition \ref{DefiMoyChamps}. Since \textsc{Zimmer}'s amenability is equivalent to almost hyperfinitness (\textit{cf}. \cite{CFW81} and \cite{AL91}), this would show the equivalence between all definitions. The reader familiar with the notion of measurable groupoid could also consult \cite[4.2.7]{AR00}.

\noindent Since we are in the measure theoretic context, we can suppose by hypothesis that there exists a Borel action $\mathbf{Z}\curvearrowleft\Omega$ such that $\R_{\mathbf{Z}}=\R$. In analogy with \cite[Thm. 2.1]{Zim78} we can define an  isometric representation of $\mathbf{Z}$ on $L^{1}(\Omega, \se B)$ by
$$(T(g)\se x)_{\omega}=\frac{dg_{\ast}(\mu)}{d\mu}(\omega)\cdot\alpha(g^{-1}\omega, \omega)x_{g^{-1}\omega}\quad\text{for every }g\in\mathbf{Z},\ \se x\in L^{1}(\Omega, \se B),$$
where $\frac{dg_{\ast}(\mu)}{d\mu}\in L^{1}(\Omega, \mathbf{R})$ is the Radon-Nikodym derivative. The adjoint representation $T_\ast:=(T^{-1})^\ast$ acts on $(L^{1}(\Omega, \se B))^\ast\simeq \widetilde L^{\infty}(\Omega, \se B^\ast)$. Given $\se\ph\in\widetilde L^{\infty}(\Omega, \se B^\ast)$ it is straightforward to see that
$$(T_\ast(g)\se\ph)_{\omega}=\alpha_{\ast}(g^{-1}\omega, \omega)\ph_{g^{-1}\omega}.$$
This means that the adjoint representation is given by the fiberwise adjoint action. Consequently $\mathbf{Z}$ acts by homeomorphisms on $\widetilde L^{\infty}(\Omega, \se B^\ast)$ with respect to the weak-$\ast$ topology and thus on $\widetilde{L}(\Omega, \se C)$ because $\ch{C}$ is supposed to be almost invariant. Since $\mathbf{Z}$ is amenable, there exists a fixed point $[\se\ph]$ and any representant of the class has the desired property.
\medskip

\noindent In the sequel we'll use the amenability in the following particular context.

\begin{exemp}\label{Exmoy}
Let $K$ be a compact metric space and consider the space $\mathcal{C}(K)^\ast_{\le1}$ endowed with the weakly-$\ast$ topology. The application
$$\begin{array}{cccl}\delta:&K&\tend&\C(K)_{\le1}^\ast\\
&x&\mapsto&[\delta_x:f\mapsto f(x)]\end{array}$$
is a homeomorphism onto its image such that $\delta_x\in\mathcal{C}(K)^\ast_{=1}$. Since $\mathcal{C}(K)_{\le1}^\ast$ is convex and weakly-$\ast$ compact by the Banach-Alaoglu Theorem, we have the equality $\mathcal{C}(K)_{\le1}^\ast=\overline{\text{co}(\{\pm\delta_x\}_{x\in K})}^{\,w\,\ast}$ by the Krein-Milman Theorem because the extremal points are $\{\pm\delta_x\}_{x\in K}$. Moreover the bijection
$$\text{Prob}(K)\simeq\{\varphi\in\mathcal{C}(K)_{\le1}^{\ast}\,|\,\varphi(\bold{1})=1\ \text{et}\
\forall f\in\mathcal{C}(K)\ (f\ge0\Rightarrow\varphi(f)\ge0)\}$$
given by the Riesz's representation Theorem allows us to consider the weakly-$\ast$ compact set of probabilities in the dual and it is well known that $\text{Prob}(K)=\overline{\text{co}(\{+\delta_x\}_{x\in K})}^{\,w\,\ast}$.

\noindent Now if $D\subseteq K$ is a dense subset, then using the continuity of $\delta$ we obtain $\C(K)_{\le1}^\ast=\overline{\text{co}(\{\pm\delta_{x}\}_{x\in D})}^{\,w\,\ast}$ and $\text{Prob}(K)=\overline{\text{co}(\{+\delta_{x}\}_{x\in D})}^{\,w\,\ast}$.

\noindent If $(\Omega, \mathcal{A},\mu)$ is a standard probability space and $\ch{K}$ is a Borel field of compact metric spaces, then the preceding results and the introduction of the Borel sections $\{\delta_{\se x}\}_{\se x\in\D}\subseteq\mathcal{L}^\infty(\Omega,\C(\se K)^\ast)_{\le 1}$ for a given fundamental family $\mathcal{D}\subseteq\mathcal{L}(\Omega, \se K)$ shows that the fields $(\Omega, \mathcal{C}(\se K)^\ast_{\le1})$ and $(\Omega, \text{Prob}(\se K))$ are Borel fields of convex, weakly-$\ast$ compact sets in the sense of Definition \ref{DefiChampcc}. A section $\se\pi$ is Borel if the corresponding section $\ph_{\se\pi}\in\mathcal{L}(\Omega,\C(\se K)^\ast)$, \textit{i.e.} $\ph_{\se\pi}(\se f)$ is Borel for every $\se f\in\mathcal L(\Omega,\C(\se K))$.

\noindent Suppose now that an action $\alpha : \R\curvearrowleft \ch{K}$ by homeomorphisms is given. The lemmas \ref{LemActionC} and \ref{LemChampBanachDual} allow us to define $\alpha^\ast(\omega', \omega) :=(\widetilde{\alpha}(\omega', \omega)^{-1})^\ast$.

\noindent The field $(\Omega, \text{Prob}(\se K))$ is obviously $\alpha^\ast$-invariant since if $\omp\mu\in\text{Prob}(\omp K)$ then $\alpha^\ast(\omega', \omega)\omp\mu$ is the image measure $\alpha(\omega', \omega)_\ast(\omp\mu)$. In particular, if $\R$ is amenable, then by definition there exists a Borel section $[\se\mu]\in\ \widetilde L(\Omega, \text{Prob}(\se K))$ which is $\alpha^\ast$-invariant.
\end{exemp}

\section{Proof of the Main Theorem}

\subsection{Fields of Convex Sets and Invariant Sections at Infinity}

Let $(\Omega,\A,\mu)$ be a standard probability space. Given a Borel field $\ch{X}$ of CAT(0) spaces we introduce the following notations:
$$\mathscr{S}:=\{[\se{C}]\mid[\se C]\text{is an invariant class of Borel subfield of non empty closed convex subsets}\},$$
$$\mathscr{M}:=\{[\se C]\in\mathscr S\mid\ch{C} \text{ is minimal for }\le\},$$

\begin{lem}\label{LemChain}
Let $(\Omega, \A, \mu)$ be a standard probability space, $\R$ an equivalence relation which quasi-preserves $\mu$, $\ch{X}$ a Borel field of proper metric spaces and an isometric action $\alpha : \R\curvearrowleft\ch{X}$.

\noindent If $\{[\se C^\beta]\}_{\beta\in\mathcal{B}}$ is a totally ordered family (\textit{i.e.} a chain) of $\mathscr{S}$, then there exists a countable family of indices $\{\beta_n\}_{n\ge1}\subseteq\mathcal{B}$ such that $[\se C^{\beta_{n+1}}]\le[\se C^{\beta_{n}}]$ for each $n\ge1$ and such that $\se C:=\cap_{n\ge1}\se C^{\beta_n}$ satisfies $[\se C]\in\mathscr S$ and $[\se C]\le[\se C^\beta]$ for all $\beta\in\mathcal{B}$.
\end{lem}

\begin{proof}
By Theorem \ref{ThmBorelSubfieldCompleteLattice} there exists a countable family of indices $\{b_n\}_{n\ge1}\subseteq\mathcal{B}$ such that the Borel subfield
$$\se C:=\cap_{n\ge1}\se C^{b_n}$$
is such that $[\se C]\le[\se C^\beta]$, $\beta\in\mathcal{B}$. By setting $\beta_n:=\min\{b_1,\ldots,b_n\}$ (where the minimum is taken for the induced order by the chain) it follows that $\se C=\cap_{n\ge1}\se C^{\beta_n}$ satisfies our conditions (the invariance of the class follows from Remark \ref{RemInvariance}).
 \end{proof}

\begin{thm}\label{Thmminempty}
Let $(\Omega, \A, \mu)$ be a standard probability space, $\R$ an ergodic equivalence relation which quasi-preserves $\mu$, $\ch{X}$ a Borel field of proper unbounded CAT(0) spaces with finite covering dimension.

\noindent If $\alpha : \R\curvearrowleft\ch{X}$ is an isometric action such that $\M=\emptyset$, then there exists an almost invariant section $\se \xi\in\mathcal{L}(\Omega, \partial\se X)$.
\end{thm}

\begin{proof} The proof proceeds in two steps.
\smallskip

\noindent (i) First we show that under the hypothesis there exists a sequence $\{[\se C^{\beta_n}]\}_{n\ge1}\subseteq[\mathscr{S}]$ such that the following conditions hold:
\begin{enumerate}
\item[(a)] $[\se C^{\beta_{n+1}}]\le[\se C^{\beta n}]$ for each $n\ge1$,
\item[(b)] $[\cap_{n\ge1}\se C^{\beta_n}]$ is the class of the empty subfield\footnote{It is the field $\se A$ defined by $\om A=\emptyset$ for every $\omega\in\Omega$.}.
\end{enumerate}
By the transposition of the Zorn's Lemma there exists a chain $\{[\se C^{\beta}]\}_{\beta\in\mathcal{B}}\in\mathscr{S}$ without lower bound in $\mathscr{S}$. So by the Lemma \ref{LemChain} applied to this chain we conclude that the Borel subfield $[\se C]:=[\cap_{n\ge1}\se C^{\beta_n}]$ is not in $\mathscr{S}$. This means that $\mu(\{\omega\in\Omega\mid C_\omega\neq\emptyset\})\neq 1$. We then show that this set is of measure null. By ergodicity the invariant set\footnote{See again Remark \ref{RemInvariance}.} $\{\omega\in\Omega\mid\om C\neq\emptyset\}$ is of measure one or null. Since the first possibility is impossible, we conclude that $[\se C]$ is the class of the empty field.
\smallskip

\noindent (ii) Consider the subchain $\{\se C^{\beta_n}\}_{n\ge1}$ and the subfield $\se C$ given in (i). There exists $\Omega'\in\A$ of measure one such that the inclusion $\om C^{\beta_n}\supseteq\om C^{\beta_{n+1}}$ holds and that $\om C=\emptyset$ for every $\omega\in\Omega'$. So by Proposition \ref{PropLimitField} we can consider the Borel field $(\Omega', \se L)\le(\Omega', \partial\se X)$ of limit sets at infinity of the subchain and this field has an unique Borel section of circumcenters
$$\se\xi:=c_{\se L}\in\mathcal{L}(\Omega', \partial\se X).$$
To prove the invariance of $[\se\xi]$ it is enough to show that $(\Omega', \se L)$ is almost-invariant. Recall that for every $\omega\in\Omega'$ and $x\in\om X$
$$\om L:=\overline{\om i\big(\{\pi_{\om C^{\beta_n}}(x)\}_{n\ge1}\big)}\cap\partial\om X.$$
So we have

\begin{eqnarray*}
\widetilde\alpha(\omega,\omega')\om L&\stackrel{\widetilde{\alpha}\text{ homeo.}}{=}&\overline{\widetilde\alpha(\omega,\omega')i_{\omega}(\{\pi_{C^{\beta_n}_{\omega}}(
x)\}_{n\ge1})}\cap\bd\omp X\\&\stackrel{\widetilde
\alpha\ \text{ext. of }\alpha}{=}&\overline{\omp
i\big(\alpha(\omega,\omega')(\{\pi_{C^{\beta_n}_{\omega}}(x)\}_{n\ge1})\big)}\cap\bd\omp X\\
&\stackrel{\alpha\ \text{isom.}}{=}&\overline{\omp
i(\{\pi_{\alpha(\omega,\omega')(C^{\beta_n}_{\omega})}(\alpha(\omega,\omega')x)\}_{n\ge1})}\cap\bd\omp
X\\
&\stackrel{\text{inv.}}{=}&\overline{\omp i\big(\{\pi_{\omp C^{\beta_n}}(\alpha(\omega,\omega')x)\}_{n\ge1}\big)}\cap\bd \omp X=\omp L.
\end{eqnarray*}
\end{proof}

\subsection{$\R$-Quasi-Invariant Sections}

\begin{defi}
Let $(\Omega,\A)$ be a standard Borel space, $\R\subseteq\Omega^2$ be a Borel equivalence relation and $\ch{X}$ be a Borel field of metrics spaces. Assume that the relation $\R$ acts on $\ch{X}$ by isometries. We say that a section $\se f\in\mathcal{S}(\Omega, \mathcal{F}(\se X))$ is $\R$-quasi-invariant (or simply invariant whenever the relation can clearly be identified) if there exists $c:\R\rightarrow\mathbf{R}$ such that
\begin{equation}\label{EqRinv}
\widetilde{\alpha}(\omega, \omega')\om f:=\om f\circ\alpha(\omega', \omega)=\omp f+c(\omega, \omega')\quad\text{for every }(\omega, \omega')\in\R.
\end{equation}
It can easily be checked that $c$ is a cocycle, \textit{i.e.} $c$ satisfies $c(\omega, \omega'')=c(\omega, \omega')+c(\omega', \omega'')$.
\end{defi}

\begin{lem}\label{LemFctQI}
Let $(\Omega,\A)$ be a standard Borel space $\R\subseteq\Omega^2$ be a Borel equivalence relation and $\ch{X}$ be a Borel field of proper CAT(0) spaces. Assume that $\R$ acts on $\ch{X}$ by isometries. If $\se f\in\mathcal{L}(\Omega, \mathcal{C}(\se X))$ is a quasi-invariant section such that $\om f$ is convex for every $\omega\in\Omega$, then the following assertions are verified.
\begin{enumerate}
\item[(i)] The sets
$$\begin{array}{rcl}\ominfini&:=&\omdansom{\inf\om f=-\infty},\\
\ominf&:=&\omdansom{\inf\om f>-\infty\text{ and is not attained}},\\
\ommin&:=&\omdansom{\inf\om f\text{ is attained}}.\end{array}$$
are Borel and invariant.
\item[(ii)] The subfield $(\se f\mid_{\ommin}-\min(\se f\mid_{\ommin}))^{-1}(\{0\})$ is a Borel subfield of non empty closed and convex sets of $(\ommin,\se X)$ which is $\R\mid_{\ommin}$-invariant.
\item[(iii)] We set $\Omega_{\inf}:=\ominfini\cup\ominf$ and we assume that $\om X$ is of finite covering dimension for every $\omega\in\Omega_{\inf}$. Then there exists a section $\se\xi\in\mathcal{L}(\Omega{_{\inf}}, \bd\se X)$ which is $\R\mid_{\Omega_{\inf}}$-invariant.
\end{enumerate}
\end{lem}

\begin{proof}
(i) The convexity assumption is not necessary to prove this assertion. The section $\inf\se f$ is Borel since by continuity $\inf\se f=\inf_{\se x\in\D}\se f(\se x)$ where $\D$ is a fundamental family and thus $\Omega_{\inf=-\infty}=(\inf \se f)^{-1}(\{-\infty\})$. Now we fix $\se x^0\in\mathcal{L}(\Omega, \se X)$. Because $\se f\in\mathcal L(\Omega,\C(\se X))$ we have
$$\Omega_{\min}=\bigcup_{R\in\mathbf{N}}(\inf\se f\mid_{B(\se x^0, R)}-\inf\se f)^{-1}(\{0\})\in\A$$
and thus $\Omega_{\inf>-\infty}=\Omega\setminus(\Omega_{\inf=-\infty}\cup\Omega_{\min})\in\A$. The invariance of these sets follows directly from the equality (\ref{EqRinv}).
\smallskip

\noindent (ii) Since $\se f$ is quasi-invariant we have $\min\om f=\min\omp f+c(\omega, \omega')$ and thus
$$\widetilde\alpha(\omega,\omega')(\om f-\min\om f)=\omp f+c(\omega,\omega')-\min\om f=\omp f-\min\omp f.$$
Consequently the section $\se{\widetilde f}\in\widetilde{\mathcal{L}}(\Omega_{\min},\C(\se X))$ defined by $\om{\widetilde f}:=\om f-\min\om f$ for every $\omega\in\Omega_{\min}$ is $\R\mid_{\Omega_{\min}}$-invariant and such that $(\se{\widetilde f})^{-1}(\{0\})$ has the required properties.
\smallskip

\noindent (iii) If $\omega\in\Omega_{\inf>-\infty}$ we define $\om{\widetilde f}:=\om f-\inf\om f$ and $\om C^n:=(\om{\widetilde f})^{-1}([0,1/n])$. The sequence $\{\se C^n\}_{n\ge1}$ satisfies the hypothesis of the Proposition \ref{PropLimitField} and thus the section of circumcenter of the limit sets at infinity is $\R\mid_{\Omega_{\inf>-\infty}}$-invariant. If $\omega\in\Omega_{\inf=-\infty}$ we consider the generalized sequence $\{\se C^\beta\}_{\beta\in\mathbf{R}}$ defined by $\om C^\beta:=\om f^{-1}(]-\infty, -\beta])$ and construct its limit set at infinity $(\Omega_{\inf=-\infty}, \se L)$. It is invariant since
\begin{align}\nonumber
\alpha(\omega,\omega')\om C^\beta&=\{\alpha(\omega, \omega')x\in\omp X\mid\om f(x)\le-\beta\}=\{y\in\omp X\mid\om f(\alpha(\omega',\omega)y)\le -\beta\}\\\nonumber
&=\{y\in\omp X\mid\omp f(y)+c(\omega,\omega')\le -\beta\}=\omp C^{\beta+c(\omega, \omega')}
\end{align}
Consequently the section of circumcenters is also $\R\mid_{\Omega_{\inf=-\infty}}$-invariant. We conclude the proof by gluing the two sections together.
\end{proof}

\noindent The following proposition is an adaptation of \cite[Lem. 2.5, p. 192]{AB98} and is a key step in the proof of the main theorem.

\begin{prop}\label{PropSecb}
Let $(\Omega,\A)$ be a Borel space, $\ch{X}$ be a Borel field of proper CAT(0) spaces, let $\se x^0\in\callo{X}$ and $\ch{B}\le(\Omega, \partial\se X)$ be a Borel subfield of closed sets.
\begin{enumerate}
\item[(i)] Assume that $\se\pi\in\mathcal{L}(\Omega, \text{Prob}(\se B))$ is fixed. For every $\omega\in\Omega$ and $x^0_\omega\in X_\omega$ we define
$$\begin{array}{cccl} \om b:&\om X&\rightarrow&\RR\\
&\om x&\mapsto&\int_{\om B}b_{\om x}(\xi)d\om\pi(\xi),
\end{array}$$
where if $x_\omega\in\om X$
$$\begin{array}{cccl} b_{\om x}:&\om B&\rightarrow&\RR\\
&\om\xi&\mapsto&b_{\om x^0,\om\xi}(\om x).
\end{array}$$
Then $\se b\in{\cal L}(\Omega,\C(\se X))$ and $\om b$ is convex for every $\omega\in\Omega$.
\item[(ii)] If $\alpha : \R\curvearrowleft\ch{X}$ acts by isometries such that $\ch{B}$ and $\se\pi$ are invariant, then the section $\se b$ is quasi-invariant.
\end{enumerate}
\end{prop}

\begin{proof}
(i) A Busemann function is convex, 1-Lipschitz and likewise is an integral of such functions. In particular $\om b$ is convex and continuous for every $\omega\in\Omega$. Moreover, by the Riesz's representation Theorem used to define the Borel structure on the field of probabilities $(\Omega, \text{Prob}(\se B))$ (\textit{cf.} Example \ref{Exmoy}) the evaluation
$$\se b(\se x)=\int_{\se B}b_{\se x}(\xi)d\se\pi(\xi)=\varphi_{\se\pi}(b_{\se x})$$
is Borel since $b_{\se x}\in{\cal L}(\Omega,\C(\se B))$ for every $\se x\in\callo X$ - see the Remark \ref{RemStructureBord}.
\smallskip

\noindent (ii) This assertion is proved by the following calculation.
\begin{eqnarray*}
(\widetilde{\alpha}(\omega,\omega')\om b)(\omp x)&=&\om b(\alompom\omp x)=\int_{\om
B}b_{\om x^0,\xi}(\alompom\omp
x)d\om\pi(\xi)\\
&\stackrel{(\ref{EquationCocycleBusemann})}{=}&\int_{\om B}\big(b_{\alompom\omp
x^0,\xi}(\alompom \omp x)-b_{\alompom\omp
x^0,\xi}(\om x^0)\big)d\om\pi(\xi)\\
&\stackrel{\alpha\text{ isom.}}{=}&\int_{\om B}b_{\omp
x^0,\alomomp\xi}(\omp x)d\om\pi(\xi)-\int_{\om B}b_{\omp
x^0,\alomomp\xi}(\alomomp\om
x^0)d\om\pi(\xi)\\
&=&\int_{\om B}b_{\omp x}(\alomomp\xi)d\om\pi(\xi)-\int_{\om
B}b_{\alomomp\om x^0}(\alomomp\xi)d\om\pi(\xi)\\
&\stackrel{\se\pi\text{ inv.}}=&\int_{\omp B}b_{\omp
x}(\xi')d\omp\pi(\xi')-\int_{\omp B}b_{\alomomp\om
x^0}(\xi')d\omp\pi(\xi')\\
&=&\omp b(\omp x)-\omp b(\alomomp\om x^0)
\end{eqnarray*}
So the function $\om b$ is quasi-invariant with $c(\omega,\omega')=b_{\omega'}(\alpha(\omega,\omega')x^0_\omega)$.
\end{proof}

\subsection{Final Proof}

\begin{proof}[Proof of Theorem \ref{ThmAH}]
Assume that the assertion $(i)$ is not satisfied. The Theorem \ref{Thmminempty} implies the existence of an almost invariant Borel subfield $\se C$ of closed convex non-empty subsets which is minimal for these properties. Without lost of generality we can assume that this field is invariant - see the Remark \ref{RemInvariance}. Let $(\Omega,\se F)$ be the Borel subfield of flat points of $(\Omega,\bd\se C)$ and $\se C\hookrightarrow\se E\times\se Y$ the Adams--Ballmann decomposition (see Lemma \ref{LemFlatPointsBorelSubfield} and Proposition \ref{PropDecompositionABborelien}). Let $\se P$ be the Borel subfield of $\bd\se C$ introduced before Lemma \ref{LemseP&seA}. Define
$$\Omega':=\{\omega\in\Omega\mid\om P\nonvide\}$$
which is a Borel and invariant subset of $\Omega$. Since $\R$ is ergodic this set is either of full or of null measure. In the first case there exists an invariant section
$$\se\xi\in{\mathcal L}(\Omega,\se P)\subseteq{\mathcal L}(\Omega,\bd\se C)\subseteq{\mathcal L}(\Omega,\bd X).$$
Indeed it is proven in \cite{AB98} that $\rad(\om P)<\pi/2$ whenever $\om P\nonvide$ and therefore the section of the circumcenters is Borel (\textit{cf}. Theorem \ref{ThmCircum}) and invariant. This contradicts the assumption made at the beginning of the proof.

\noindent We can therefore assume that $\Omega\setminus\Omega'$ is of full measure and we won't lose generality if we assume that $\Omega'=\emptyset$. Then $\om C=\om E\times\om Y$ where $\om E$ is a finite dimensional space and $\bd\om Y$ doesn't contain flat points for all $\omega\in\Omega$. The Borel field $\bd\se Y$ is invariant by the property (iii) of the Adams--Ballmann decomposition. We set
$$\Omega'':=\{\omega\in\Omega\mid \bd\om Y=\emptyset\}$$
which is an invariant and Borel subset of $\Omega$ and therefore of full or of null measure. Assume first that it is of full measure. Then $(\Omega'',\se Y)$ is a Borel field of bounded CAT(0) spaces and therefore the section of the circumcenters $c_{\se Y}\in{\cal L}\big(\Omega'',\se Y\big)$ is Borel and invariant (\textit{cf.} Lemma \ref{LemCentresBorel}). Thus $(\Omega'',\se E\times\{c_{\se Y}\})$ is a Borel subfield of flats of $\se C$ (and thus of $\se X$) which is invariant. Ergodicity obviously implies that the dimension is essentially constant.

\noindent So we can assume that $\Omega''=\emptyset$. Since the relation is amenable there exists an invariant Borel section of probabilities $\se{\pi}\in\mathcal L(\Omega,\text{Prob}(\bd\se Y))$. By the Proposition \ref{PropSecb} the section of convex functions
$$\se b\in{\cal S}(\Omega,\F(\se X))$$
defined, for $\se x^0\in\mathcal{L}(\Omega, \se C)$ fixed, by
$$\om b(x)=\int_{\bd\om Y}b_{\xi,\om x^0}(x)d\om\pi(\xi)$$
is such that $\se b\in{\cal L}(\Omega,\C(\se X))$ and that $\om b$ is convex for every $\omega\in\Omega$. Moreover it is quasi-invariant. If we define $\om f$ to be the restriction of $\om b$ to $\om C$, then $\se f\in{\cal L}(\Omega,\C(\se C))$ is again quasi-invariant. Thus by Lemma \ref{LemFctQI} the subsets
$$\begin{array}{lcl}
\Omega_{\inf}&:=&\{\omega\in\Omega\mid\inf\om f\text{ is not attained}\}\\
\Omega_{\min}&:=&\{\omega\in\Omega\mid\inf\om f\text{ is attained}\}
\end{array}$$
are invariant Borel subsets of $\Omega$. By ergodicity one of it has to be of full measure. Let's prove that $\mu(\Omega_{\min})=1$ is not possible. In that case, $\se B:=\big(\se b\mid_{\Omega_{\min}}-\min(\se b\mid_{\Omega_{\min}})\big)^{-1}(\{0\})$ would be an invariant Borel subfield of closed convex subsets and - by minimality of $\se C$ - $\om B=\om C$ would hold for almost every $\omega\in\Omega$. This would mean that $\om b$ is a constant function and that the points of $\bd\om Y$ in the support of $\om\pi$ are flats. This contradicts the construction of $\om Y$. Therefore $\Omega_{\inf}$ has to be of full measure. But we have shown in Lemma \ref{LemFctQI} that in this case we can construct an invariant section $\se\xi\in{\cal L}(\Omega_{\inf},\bd\se Y)$ and that contradicts the original assumption of the proof.
\end{proof}

\section{Concluding Remarks}

\noindent It is very likely that our result might hold for any amenable Borel groupoid. At least it has been checked for amenable $G$-spaces (see \cite{And10} or \cite{Duc11}).

\noindent Our result will be used in two forthcoming articles, one from each author.
\bigskip

\noindent M. A. , Section de Mathématiques de l'Université de Genève, \texttt{martin.anderegg@unige.ch }

\noindent P. H. , \'Ecole Polytechnique Fédérale de Lausanne, \texttt{philippe.henry@epfl.ch}

\small

\end{sloppypar}
\end{document}